\documentclass[11pt]{amsart}
\usepackage{amssymb,mathrsfs,graphicx,enumerate}
\usepackage{amsmath,amsfonts,amssymb,amscd,amsthm,bbm}
\usepackage{kotex}
\usepackage[retainorgcmds]{IEEEtrantools}
\usepackage{colortbl}
\usepackage{tikz}
\usetikzlibrary{matrix}
\allowdisplaybreaks
\DeclareMathOperator*{\argmax}{arg\,max}
\title[The kinetic Lohe Hermitian sphere model]{Emergent behaviors of the kinetic Lohe Hermitian sphere model}

\author[Byeon]{Junhyeok Byeon}
\address[Junhyeok Byeon]{\newline Department of Mathematical Sciences\newline Seoul National University, Seoul 08826, Republic of Korea}
\email{giugi2486@snu.ac.kr}

\author[Ha]{Seung-Yeal Ha}
\address[Seung-Yeal Ha]{\newline Department of Mathematical Sciences and Research Institute of Mathematics \newline Seoul National University, Seoul 08826 and \newline
Korea Institute for Advanced Study, Hoegiro 85, 02455, Seoul, Republic of Korea}
\email{syha@snu.ac.kr}

\author[Hwang]{Gyuyoung Hwang}
\address[Gyuyoung Hwang]{\newline Department of Mathematical Sciences\newline Seoul National University, Seoul 08826, Republic of Korea}
\email{hgy0407@snu.ac.kr}

\author[Park]{Hansol Park}
\address[Hansol Park]{\newline Department of Mathematical Sciences\newline Seoul National University, Seoul 08826, Republic of Korea}
\email{hansol960612@snu.ac.kr}

\newtheorem{theorem}{Theorem}[section]
\newtheorem{lemma}{Lemma}[section]
\newtheorem{corollary}{Corollary}[section]
\newtheorem{proposition}{Proposition}[section]
\newtheorem{remark}{Remark}[section]

\newtheorem{definition}{Definition}[section]

\newcommand{\bbr}{\mathbb R}

\newcommand{\bbs}{\mathbb S}
\newcommand{\bbz}{\mathbb Z}

\newcommand{\bbc}{\mathbb C}

\makeatletter
\@namedef{subjclassname@2020}{\textup{2020} Mathematics Subject Classification}
\makeatother

\begin{document}

\date{\today}


\thanks{\textbf{Acknowledgment.} The work of S.-Y. Ha was supported by National Research Foundation of Korea(NRF-2020R1A2C3A01003881).}

\begin{abstract}
We study a global well-posedness of measure-valued solutions to the kinetic Lohe Hermitian sphere(LHS) model derived from the Lohe tensor(LT) model on the set of rank-1 complex tensors(i.e. complex vectors) with the same size and investigate emergent behaviors. The kinetic LHS model corresponds to a complex analogue of the kinetic LS model which has been extensively studied in the literature on the aggregation modeling of Lohe particles on the unit sphere in Euclidean space. In this paper, we provide several frameworks in terms of system parameters and initial data leading to the local and global well-posedness of measure-valued solutions. In particular, we show emergent behaviors of the kinetic LHS model with the same free flows by analyzing the temporal evolution of the order parameter. 
\end{abstract}
\maketitle \centerline{\date}


\section{Introduction} \label{sec:1}
\setcounter{equation}{0} 
Emergent behaviors are ubiquitous in biological complex systems, e.g., aggregation of bacteria \cite{T-B-L, T-B}, schooling of fish, flocking of birds, the synchronous firing of fireflies, and neurons \cite{A-B, B-B, Pe} etc. These collective phenomena were first modeled by two pioneers, Arthur Winfree \cite{Wi2} and Yoshiki Kuramoto \cite{Ku2}, almost a half-century ago. Since then, several mathematical models were proposed and studied from various points of view. Among them, our main interest lies in the LHS(Lohe Hermitian sphere) model \cite{H-P2} corresponding to the special case of the LT(Lohe tensor) model \cite{H-P1}. The LT model is a higher-dimensional extension of low-dimensional aggregation models such as the Kuramoto model \cite{A-B, B-C-M, C-H-J-K, C-S, D-X, D-B1, D-B, H-K-R}, swarm sphere models \cite{C-H5, HKLN, J-C, Lo-1, Lo-2, M-T-G, T-M, Zhu}, and matrix models \cite{B-C-S, D-F-M-T, De, Kim, Lo-0} (see survey articles and books \cite{A-B-F, B-H, D-B1, P-R, St, VZ, Wi1}). Before we move on further, we introduce some notations:
\begin{align}
\begin{aligned} \label{Not}
& z = (z^1, \cdots,z^{d}) \in \bbc^{d}, \quad w = (w^1, \cdots,w^{d}) \in \bbc^{d}, \quad  \langle w, z \rangle := \sum_{i = 1}^{d} \overline{w^i} z^i, \\
& \|z \| := \sqrt{ \langle z, z \rangle}, \quad  \mathbb{HS}^{d-1} := \{ z \in \mathbb{C}^{d} \, | \, \|z\| = 1\}, \quad \mathcal{N} := \{ 1, 2, \cdots, N \},
\end{aligned}
\end{align}
where $\overline{z^k}$ is the complex conjugate of $z^k\in\bbc$.\newline

Consider the LHS model \cite{H-P2}:
\begin{equation}  \label{A-0}
\displaystyle  \dot{z}_j= \Omega_j z_j+\displaystyle\frac{\kappa_0}{N}\sum_{k=1}^N \big(\langle z_j, z_j\rangle z_k-\langle z_k, z_j\rangle z_j\big) +\frac{\kappa_1}{N}\sum_{k=1}^N \big(\langle z_j, z_k\rangle-\langle z_k,z_j\rangle\big)z_j,
\end{equation}
where $\kappa_{0}$ and $\kappa_{1}$ are constant coupling gains coined as  ``{\it Lohe sphere coupling gain}" and ``{\it rotational coupling gain}", respectively.  Here, $\Omega_j$ is a $d \times d$ skew-Hermitian matrix:
\[ \Omega_j^\dagger = -\Omega_j, \quad j \in \mathcal{N}, \]
where $\Omega_j^\dagger$ is the Hermitian conjugate of $\Omega_j$. \newline

In this paper, we are interested in the situation in which the number of particles is sufficiently large so that system \eqref{A-0} can be effectively approximated by the corresponding mean-field model. To be more precise, we set  a phase space $\Xi$:
\[
\Xi:=\mathbb{HS}^{d-1}\times \mathrm{Skew}_{d}\mathbb{C},\quad \mathrm{Skew}_d\bbc :=\{A\in\bbc^{d\times d}: A^\dagger=-A\}, 
\]
and  let $f = f(t, z, \Omega)$ be the one-particle distribution function of the Lohe infinite ensemble at the phase space point $(z, \Omega)$ at time $t$. Then, by the standard BBGKY hierarchy (see Appendix \ref{App-C}), the spatial-temporal dynamics of the kinetic density function $f$ is governed by the kinetic LHS model:
\begin{align}\label{A-1}
\begin{cases}
\partial_t f+\nabla_z \cdot(L[f]f)=0, \quad t > 0,\quad (z, \Omega)\in\Xi, \\
\displaystyle L[f](t,z, \Omega)=\Omega z +\int_{\Xi}[\kappa_0(z_*-\langle{z_*, z}\rangle z)+\kappa_1(\langle{z, z_*}\rangle z-\langle{z_*, z}\rangle z)]f(t,z_*, \Omega_*)d\sigma_{z_*}d\Omega_*.
\end{cases}
\end{align}

\vspace{0.2cm}

For this kinetic model, we are interested in the following two questions:
\begin{itemize}
\item
(Q1)~(Global well-posedness): under what conditions on system parameters and initial data, can we guarantee a global well-posedness of measure-valued solutions? 
\vspace{0.2cm}
\item
(Q2)~(Eemergent dynamics): under what framework, can we show collective behaviors?
\end{itemize}

The main purpose of this paper is to answer the above two questions (Q1) and (Q2). In fact, our main results can be summarized as follows.
 \newline

First, we provide an improved exponential aggregation estimate for the LHS model \eqref{A-0} under the following system parameters and initial data (see Theorem \ref{T2.3}):
\begin{equation}\label{A-1-1}
|\kappa_1| < \frac{\kappa_0}{2} \quad \mbox{and} \quad \max_{1 \leq k, l \leq N}|1-\langle z_k^0, z_l^0\rangle| < 1 - \frac{2|\kappa_1|}{\kappa_0}-\delta,
\end{equation}
for a constant $\delta\in\left(0, 1-\frac{2|\kappa_1|}{\kappa_0}\right)$, and let $Z=(z_1, \cdots, z_N)$ be a solution to \eqref{LHS-a}. Then, one has 
\[
\max_{1 \leq k, l \leq N}\|z_k(t)-z_l(t) \| \leq C(Z^0) \exp(-\kappa_0\delta t), \quad t \geq 0.
\]
Note that the condition \eqref{A-1-1} allows $\kappa_1$ to be negative which is different from earlier results in \cite{H-P1}.

Second, we provide a global well-posedness of measure-valued solutions to \eqref{A-1}. More precisely, if system parameters and initial measure $\mu_0$ satisfy 
\[
|\kappa_1| < \frac{\kappa_0}{2},\quad 0\leq\sup_{z, w \in \mathrm{supp}(\mu_0)}|1-\langle z, w\rangle| < 1 - \frac{2|\kappa_1|}{\kappa_0}-\delta,\quad  \Omega_j\equiv \Omega,
\]
for some positive constant $\delta$, then there exists a unique measure-valued solution to \eqref{A-1} in the whole time interval $[0, \infty)$ satisfying 
\[
 \lim_{N \to \infty}\sup_{0 \leq t < \infty} W_2(\mu_t, \mu^N_t)=0.
\]
where $W_2$ is the Wasserstein-2 distance (see Theorem \ref{T3.1}), and $\mu_t^N$ is an empirical measure. For measure-valued solutions $\mu$ and $\nu$ to \eqref{A-1} with the initial data $\mu_0$ and $\nu_0$, respectively, we derive a finite-time stability estimate (Corollary \ref{C3.1}):
\[ 
W_p(\mu_t, \nu_t)\leq C(T) \cdot W_p(\mu_0, \nu_0), \quad t \in [0,T).
\] 

Third, we consider the case in which natural frequency matrices are the same, say $\Omega_j = \Omega$. In this case, thanks to the solution splitting property (see Proposition \ref{P4.1}), we can set
\[ \rho (t, z) = \rho_t(z) := \pi \# f_t(z, \Omega), \]
where $\pi$ is the projection map $\pi : \mathbb{HS}^{d-1}\times \mathrm{Skew}_{d}\mathbb{C} \rightarrow \mathbb{HS}^{d-1}$. Then, $\rho$ satisfies the following continuity equation with a nonlocal flux on $\bbr_+ \times \mathbb{H}\bbs^{d-1}$: 
\begin{align}\label{A-2}
\begin{cases}
\partial_t \rho+\nabla_z \cdot(L[\rho]\rho)=0, \quad t > 0, \quad z\in\mathbb{HS}^{d-1}, \\
\displaystyle L[\rho](t,z)= \int_{\mathbb{HS}^{d-1}}[\kappa_0(z_*-\langle{z_*, z}\rangle z)+\kappa_1(\langle{z, z_*}\rangle z-\langle{z_*, z}\rangle z)]\rho(t,z_*)d\sigma_{z_*}\\
\rho(0,z)=\rho_0(z).
\end{cases}
\end{align} 

Fourth, we deal with the emergent dynamics of kinetic system \eqref{A-2}. More precisely, let $\rho$ be a smooth solution of system \eqref{A-2} and coupling strengths satisfy
\[
\kappa_0>0,\quad \kappa_0+2\kappa_1\geq0.
\]
We introduce an order parameter $R$:
\[
R(t) =\left|\int_{\mathbb{HS}^{d-1}}z\rho(t,z)d\sigma_z\right|.
\]
Note that $R$ denotes the modulus of the averaged points(centroid). By direct estimates in Proposition \ref{P4.1}, one can derive monotonicity of $R^2$ and uniform bound for the second derivatives of $R^2$.
\[ \frac{dR^2}{dt} \geq 0 \quad \mbox{and} \quad \sup_{0 \leq t < \infty} \Big| \frac{d^2 R^2}{dt^2}  \Big | < \infty. \]
Then, thanks to Barbalat's lemma \cite{Ba}, one obtains the following emergent estimate (see Theorem \ref{T4.1}):
\[
\lim_{t\to\infty} \int_{\mathbb{HS}^{d-1}}(\|J_\rho\|^2-|z\cdot J_\rho|^2)\rho(t,z)d\sigma_z=0, \quad J_\rho := \int_{\mathbb{HS}^{d-1}}z\rho(t,z)d\sigma_z.
\]
Note that for a fixed $t > 0$, 
\begin{align*}
\int_{\mathbb{HS}^{d-1}}(\|J_\rho\|^2&-|z\cdot J_\rho|^2)\rho(z)d\sigma_z=0 \\
&\Longleftrightarrow \quad (\|J_\rho\|^2-|z\cdot J_\rho|^2) \equiv 0 \quad \text{ a.e. in supp}(\rho),
\end{align*}
and meanwhile unit modulus of $z$ and Cauchy-Schwarz inequality imply
\[
| z \cdot J_\rho |^2 \leq \|z\|^2\|J_{\rho}\|^2 = \|J_{\rho}\|^2,
\]
where the equality holds if and only if $z$ and $J_\rho$ are parallel. Therefore, an emergent estimate indicates that every $z \in \text{supp}(\rho)$ tends to be parallel to $J_\rho$, which means the emergence of either bi-polar state or complete aggregation (see Definition \eqref{measure_cg}).

\vspace{0.5cm}

The rest of this paper is organized as follows. In Section \ref{sec:2}, we briefly review emergent dynamics of the LHS model and the kinetic LS model which corresponds to the real counterpart of the kinetic LHS model \eqref{A-1}. In Section \ref{sec:3}, we study a global well-posedness of measure-valued solutions to the kinetic LHS model via the uniform stability estimate with respect to initial data, and then we employ a standard particle-in-cell method together with the uniform stability estimate, we derive a global well-posedness of measure-valued solutions. In Section \ref{sec:4}, we provide emergent estimates for the continuity equation \eqref{A-2} by analyzing the temporal evolution of the order parameter. Finally, Section \ref{sec:5} is devoted to a summary of our main results and some remaining issues for future work. In the appendix, we provide proofs for Theorem \ref{T2.3}, Proposition \ref{Lp_stability}, a formal BBGKY hierarchy for the derivation of the kinetic LHS model, and proof of Lemma 4.3.

\vspace{0.5cm}

\noindent \textbf{Notation.} For complex vectors $w = (w^1, \cdots, w^d), z = (z^1, \cdots, z^d) \in\bbc^d$, we use two inner-product type operations:
\begin{equation} \label{A-3}
\langle w, z\rangle :=\sum_{i=1}^d \overline{w^i} z^i, \quad w\cdot z :=\sum_{i=1}^d\left( \mathrm{Re}w^i \mathrm{Re}z^i + \mathrm{Im}w^i  \mathrm{Im}z^i \right).
\end{equation}
Note that $\langle \cdot, \cdot \rangle$ and $~\cdot~$ are complex-valued and real-valued functions, respectively. Throughout the paper,  we also use handy notations:
\[ \max_{i,j} := \max_{1 \leq i, j \leq N},  \quad \max_{i} := \max_{1 \leq i \leq N}.  \]

\setcounter{equation}{0}
\section{Preliminaries}\label{sec:2}
In this section, we present the Lohe Hermitian sphere(LHS) model and review the kinetic Lohe sphere(LS) model and their basic properties in relationship with well-posedness and emergent dynamics. 

\subsection{The LHS model}  \label{sec:2.1}
First, we begin with the Lohe tensor(LT) model \cite{H-P1} which is the first-order aggregation model on the space of rank-$m$ complex tensors with the same size following the presentations in \cite{H-P1, H-P2}. This incorporates aforementioned low-rank aggregation models such as  the Kuramoto model, the Lohe sphere model, and the Lohe matrix model. 

For a rank-$m$ complex tensor $T$ with size $d_1\times d_2\times \cdots \times d_m$, we denote $[T]_{\alpha_1 \cdots \alpha_m}$ by the $(\alpha_1, \cdots, \alpha_m)$-th component of $T$. We also set $\bar{T}$ to be the rank-$m$ with tensor whose components are the complex conjugate of the componentess in $T$:
\[ [\bar{T}]_{\alpha_1 \cdots \alpha_m} =\overline{[T]_{\alpha_1 \cdots \alpha_m}}. \]

Let ${\mathcal T}_m(\bbc; d_1 \cdots d_m)$ be the space of all rank-$m$ complex tensors with size $d_1 \cdots d_m$,  For simplicity of presentation, we introduce several handy notation as follows: for $T \in {\mathcal T}_m(\bbc; d_1 \times \cdots\times d_m)$ and $A \in  {\mathcal T}_{2m}(\bbc; d_1 \times\cdots\times  d_m \times d_1 \times \cdots\times d_m)$, we set
\begin{align*}
\begin{aligned}
& [T]_{\alpha_{*}}:=[T]_{\alpha_{1}\alpha_{2}\cdots\alpha_{m}}, \quad [T]_{\alpha_{*0}}:=[T]_{\alpha_{10}\alpha_{20}\cdots\alpha_{m0}},  \quad  [T]_{\alpha_{*1}}:=[T]_{\alpha_{11}\alpha_{21}\cdots\alpha_{m1}}, \\
&  [T]_{\alpha_{*i_*}}:=[T]_{\alpha_{1i_1}\alpha_{2i_2}\cdots\alpha_{mi_m}}, \quad [T]_{\alpha_{*(1-i_*)}}:=[T]_{\alpha_{1(1-i_1)}\alpha_{2(1-i_2)}\cdots\alpha_{m(1-i_m)}}, \\
&  [A]_{\alpha_*\beta_*}:=[A]_{\alpha_{1}\alpha_{2}\cdots\alpha_{m}\beta_1\beta_2\cdots\beta_{m}}, \quad T_c := \frac{1}{N} \sum_{k=1}^{N} T_k.
\end{aligned}
\end{align*}
Now, we are ready to provide the LT model. Consider collections $\{T_j \}_{j=1}^{N}$ and  $\{A_j\}_{j=1}^{N}$ consisting of $N$-collection of rank-$m$ complex tensors with size $d_1 \cdots d_m$ and skew-Hermitian rank-$2m$ complex tensors with  size $(d_1 \times \cdots \times d_m) \times (d_1 \times \cdots \times d_m)$ satisfying 
\[
[\bar{A}_j]_{\alpha_{*0}\alpha_{*1}} = -[A_j]_{\alpha_{*1}\alpha_{*0}},
\]
for all $\alpha_{*0}, \alpha_{*1}$. \newline

Then, the component-wise form of LT model can be expressed as follows:
\begin{align}\label{M-1}
\begin{cases}
\dot{[T_j]}_{\alpha_{*0}} = [A_j]_{\alpha_{*0}\alpha_{*1}}[T_j]_{\alpha_{*1}}
+\displaystyle\sum_{i_* \in \{0,1\}^m} \kappa_{i_*}\left( [T_c]_{\alpha_{*i_{*}}} \bar{[T_j]}_{\alpha_{*1}}[T_j]_{\alpha_{*(1-i_*)}}
- [T_j]_{\alpha_{*i_{*}}}\bar{[T_c]}_{\alpha_{*1}}[T_j]_{\alpha_{*(1-i_*)}} \right), \\
T_j(0) = T_j^0, \quad \|T_j^0\|_{\mathrm{F}}=1, \quad j\in\mathcal{N},\quad t>0,\\
\end{cases}
\end{align}
where $\{\kappa_{i_*}\}$ is a set of nonnegative coupling strengths, and $\| \cdot \|_{\mathrm{F}}$ is the Frobenius norm defined as follows:
\[
\|T\|_{\mathrm{F}} := \left( \sum_{\alpha_1,\alpha_2,\cdots,\alpha_m} \Big| [T]_{\alpha_1\alpha_2\cdots\alpha_m} \Big|^2 \right)^{\frac{1}{2}}.
\]
Here, $A_j$ is a rank-$2m$ complex tensor with size $(d_1\times d_2\times \cdots \times d_m) \times (d_1\times d_2\times \cdots \times d_m)$ satisfying the relation:
\[
[\bar{A}_j]_{\alpha_{*0}\alpha_{*1}} = -[A_j]_{\alpha_{*1}\alpha_{*0}},
\]
for all $\alpha_{*0}, \alpha_{*1}$. \newline

 Note that the first and second terms in the R.H.S. of $\eqref{M-1}_1$ represent a free flow and cubic aggregation couplings, respectively. Despite its structural complexity of \eqref{M-1}, system \eqref{M-1} can exhibit emergent collective behaviors under suitable frameworks (see \cite{H-P1}). Among others, one of our main purpose of this paper is to derive a sufficient framework for complete aggregation of the LT model:
\begin{align}\label{M-2}
\lim_{t \to \infty} \max_{i,j} \| T_i(t) - T_j(t) \|_{\mathrm{F}} = 0,
\end{align}
under the special situation:
\[ T_j = z_j \in {\mathcal T}_1(\bbc; d) = \bbc^d, \quad \mbox{and} \quad A_j = \Omega_j \in {\mathcal T}_2(\bbc; d \times d) = \bbc^{d \times d}. \] 
In this setting, system \eqref{M-1} reduces to the \emph{Lohe  Hermitian sphere}(LHS) model:
\begin{equation}\label{M-3}
\begin{cases}
\dot{z}_j=\Omega_j z_j+\kappa_0 \Big (\langle{z_j, z_j}\rangle z_c-\langle{z_c, z_j}\rangle z_j \Big )+\kappa_1 \Big (\langle{z_j, z_c}\rangle-\langle{z_c, z_j}\rangle \Big )z_j, \quad t > 0, \vspace{0.2cm}\\
z_j(0)=z_j^0, \quad  \|z^0_j \| = 1, \quad j \in {\mathcal N},
\end{cases}
\end{equation}
where { {$z_c :=\frac{1}{N}\sum_{k=1}^Nz_k$}}, and $\Omega_j$ is a skew-Hermitian $d \times d$ matrix.  It is shown in \cite{H-P2} that system \eqref{M-3} preserves the modulus of $z_j$:
\[ \|z_j(t) \| = 1, \quad  t > 0, \quad j \in {\mathcal N}. \]
Next, we present results from \cite{H-P2} on complete aggregation for the LHS model without proofs for the comparison with result to be presented in later sections.
\begin{theorem} 
\emph{(Complete aggregation) \cite{H-P2}} 
Suppose system parameters and initial data satisfy
\[
0<\kappa_1<\frac{1}{4}\kappa_0, \quad \|z_c(0)\|>\frac{N-2}{N}, \quad \Omega_j = \Omega, \quad j \in {\mathcal N},
\]
and let $\{z_j\}$ be a global solution to \eqref{M-3}. Then complete aggregation emerges asymptotically:
\[  \displaystyle \lim_{t \to \infty} \max_{i,j}\|z_i(t)-z_j(t) \| = 0. \]
\end{theorem}
\begin{proof} Detailed proof can be found in Theorem 4.1 of \cite{H-P2}.
\end{proof}

Next, we present an exponential aggregation estimate for the LHS model with the same free flow $\Omega_j = \Omega$ with a relaxed sign condition on $\kappa_1$. For this, we define functionals $\mathcal{F}$ and ${\mathcal G}$ as follows:
\[
\mathcal{F}:= \max_{k,l}\left| 1 - \langle z_k, z_l \rangle \right|, \quad  \mathcal{G}:=\max_{k, l}\|z_k-z_l\|.
\]
It follows from the following inequality:
\[
\|z_i-z_j\|^2=2\mathrm{Re}(1-\langle z_i, z_j\rangle)\leq 2|1-\langle z_i, z_j\rangle|
\]
that 
\begin{equation} \label{M-3-1}
\mathcal{G} \leq 2\sqrt{\mathcal{F}}.
\end{equation}
For complete aggregation, it suffices to show the exponential decay of ${\mathcal F}$. Note that for a homogeneous ensemble, without loss of generality, we may assume $\Omega = 0$ (thanks to the solution splitting property) so that $z_i$ satisfies 
\begin{align}\label{LHS-a}
\dot{z}_i = \frac{\kappa_0}{N}\sum_{k=1}^N (z_k-\langle z_k, z_i \rangle z_i)+ \frac{\kappa_1}{N}\sum_{k=1}^N (\langle z_i, z_k \rangle - \langle z_k, z_i \rangle)z_i.
\end{align}
\begin{theorem}[Exponential aggregation]\label{T2.3}
Suppose system parameters and initial data satisfy
\[
|\kappa_1| < \frac{\kappa_0}{2},\quad 0\leq\max_{k, l}|1-\langle z_k^0, z_l^0\rangle| < 1 - \frac{2|\kappa_1|}{\kappa_0}-\delta,
\]
for some positive constant $\delta$, and let $Z=(z_1(t), \cdots, z_N(t))$ be a solution to \eqref{LHS-a}. Then, one has 
\begin{equation} \label{M-3-2}
\mathcal{F}(t)\leq \mathcal{F}^0\exp\left(-2\kappa_0\delta t\right), \quad \mathcal{G}(t) \leq 2\sqrt{\mathcal{F}^0}\exp(-\kappa_0\delta t).
\end{equation}
\end{theorem}
\begin{proof}
Since the proof is rather lengthy, we leave it to Appendix \ref{App-A}, whereas the second estimate in \eqref{M-3-2} follows from the first estimate and \eqref{M-3-1} directly. 
\end{proof}
\begin{remark} Note that the framework in this theorem allows the coupling strength $\kappa_1$ to be negative (i.e. repulsive coupling). 
\end{remark}

\subsection{The LS model} \label{sec:2.2}
Consider the LS model on $\bbr^d$ for rank-1 real tensors which corresponds to the real counterpart of the LHS model \eqref{M-3} for $z_j = x_j \in \bbr^d.$ In this case, one has
\[ \langle{z_j, z_c}\rangle-\langle{z_c, z_j}\rangle  = 0. \]
Thus, system \eqref{M-3} can be reduced to 
\begin{equation}\label{M-4}
\begin{cases}
\displaystyle \dot{x}_j=\Omega_j x_j+\kappa_0 \big(\langle{x_j, x_j}\rangle x_c-\langle{x_c, x_j}\rangle x_j \big), \quad t >0, \vspace{0.2cm}\\
\displaystyle  x_i(0)=x_i^0\in \mathbb{S}^{d-1}, \quad i \in \mathcal{N}.
\end{cases}
\end{equation}
Note that once the initial data $x^0_i$ lie in $\bbs^{d-1}$, one can see that system \eqref{M-4} conserves the unit modulus. Thus, system \eqref{M-4} can be regarded as an aggregation model on 
the unit sphere $\bbs^{d-1}$. \newline

\noindent For a homogeneous ensemble with $\Omega_k=\Omega$ for all $k \in {\mathcal N}$,  system \eqref{M-4} enjoys the solution splitting property(see Lemma 2.2 and Proposition 2.3 of 
\cite{H-P1}, respectively). Without loss of generality, we can assume $\Omega = 0$ and system \eqref{M-4} can be rewritten as 
\begin{align}\label{M-4.5}
\dot{x}_j=\kappa_0 \big(x_c-\langle{x_c, x_j}\rangle x_j \big)=\kappa_0 \mathbb{P}_{x_j^{\perp}}x_c,
\end{align}
where $\mathbb{P}_{x^{\perp}}$ is a projection on the orthogonal complement of a unit vector $x$ given as follows:
\[
\mathbb{P}_{x^{\perp}}u = u-\langle u, x \rangle x, \quad u \in \bbr^d.
\]
Since $\kappa_0$ is linear in \eqref{M-4.5}, we may set $\kappa_0=1$ by time scaling if necessary. Thus, $x_j$ satisfies 
\begin{align}\label{M-5}
\dot{x}_j=\mathbb{P}_{x_j^{\perp}}x_c.
\end{align}
Since the R.H.S. of \eqref{M-5} belongs to $T_{x_j} \bbs^{d-1}$,  system \eqref{M-5} can be regarded as a \emph{particle} model on $\bbs^{d-1}$, as its governing equation describes the interaction between particles on the unit sphere. Let $\rho(t,\cdot) $ be a probability measure on $\bbs^{d-1}$. Then, by the BBGKY hierarchy for the particle model \eqref{M-5}, one can derive the \emph{continuum aggregation model}:
\begin{align}\label{M-6}
\partial_t \rho + \nabla_x \cdot (\rho\mathbb{P}_{x^{\perp}}J_\rho )=0, \quad \text{with} \quad J_\rho =\int_{\mathbb{S}^{d-1}}x\rho \mathrm{d}\sigma_x,
\end{align}
where $\nabla_x$ and $\mathrm{d}\sigma_x$ are divergence operator and standard surface measure on $\mathbb{S}^{d-1}$, respectively. Here, $\rho(t,\cdot)$ can be understood as a \emph{measure-valued} extension of $\mathcal{X}(t)=\{x_i(t)\}$ in the sense that 
\begin{enumerate}
\item For a dirac mass $\delta_x$ concentrated at $x$, we have
\begin{align*}
\begin{aligned}
& \mathcal{X}(t)=\{x_i(t)\} \text{ is a solution of } \eqref{M-5} \\
& \Longleftrightarrow\quad \mbox{the empirical mesure}~\rho(t):=\frac{1}{N}\sum_{i=1}^N \delta_{x_i(t)} \text{ is a measure-valued solution of } \eqref{M-6}.
\end{aligned}
\end{align*}

\vspace{0.2cm}

\item If $\rho$ is a probability density function of $x$, then $J_\rho =\int_{\mathbb{S}^{d-1}}x\rho \mathrm{d}\sigma_x$ can be regarded as the expected value $\mathbb{E}[x;\mathbb{S}^{d-1}]$. In this sense, $J_\rho$ is a generalized concept of a centroid $x_c$ for the particle model \eqref{M-5}.
\end{enumerate}
Now we present an analogue of \eqref{M-2}. We first provide definitions of Wasserstein spaces and distances. 

\begin{definition}
Let $(\Xi,\| \cdot \|)$ be a normed space, and $p \in [1, \infty)$. 
\begin{enumerate}
\item The Wasserstein space of order $p$ on $\Xi$ is defined as a collection of probability measures with a finite $p$-th moment:
\[
\mathcal{P}_p(\Xi):=\left\{ \mu\in\mathcal{P}(\Xi): \langle{\mu, \|y\|^p}\rangle=\int_{\Xi}\|y\|^p\mu(dy)<\infty \right\}.
\]
\vspace{0.2cm}

\item Let $\mu$ and $\nu$ in $\mathcal{P}_p(\Xi)$ be two measures. Then, the Wasserstein metric $W_p$ of order $p$ between $\mu$ and $\nu$ is given by
\[
W_p(\mu, \nu):=\left(\inf_{\gamma\in\Pi(\mu, \nu)}\int_{\Xi\times\Xi}\|y-\tilde{y}\|^p\gamma(dyd\tilde{y})\right)^{\frac{1}{p}},
\]
where $\Pi(\mu, \nu)$ is the collection of probability measures on $\Xi\times\Xi$ with marginals $\mu$ and $\nu$. Such $\gamma\in\Pi(\mu, \nu)$ are called the transport plans, and those achieving the infimum, if any, are called the optimal transport plans.
\end{enumerate}
\end{definition}
\begin{remark}\label{Wp_remark}
If $(X,d)$ is a Polish space and $p \in [1, \infty)$,  $W_p$-metric metrizes the weak convergence in $\mathcal{P}_p(X)$(See Theorem 6.9 of \cite{Vi}). In other words, if $(\mu_n)$ is a sequence of measures in $\mathcal{P}_p(X)$, then for $\mu \in \mathcal{P}_p(X)$,
\[
\mu_k \text{ converges weakly to } \mu \text{ in } \mathcal{P}_p(X) \text{ as } k \to \infty  \quad \Longleftrightarrow \quad W_p(\mu_k, \mu) \rightarrow 0 \text{ as } k \to \infty,
\]
and this justifies to equip $W_p$ distance on $\mathcal{P}_p(\Xi)$. Furthermore, this directly implies the continuity of $W_p(X)$; if $(\mu_n)$ and $(\nu_n)$ converge weakly to $\mu$ and $\nu$ in $\mathcal{P}_p(X)$, respectively,
\[
W_p(\mu_k,\nu_k) \rightarrow W_p(\mu,\nu).
\]
\end{remark}


As an analogue of \eqref{M-2}, one can verify the following result.

\begin{theorem}\label{T2.4}
\emph{\cite{F-L}}
Let $\rho_0$ be a probability measure on $\mathbb{S}^{d-1} \subset \bbr^d$, and let $\rho \in C\left(\bbr_+, \mathcal{P}(\bbs^{d-1})\right)$ be a solution of \eqref{M-6} with initial condition $\rho(0,x)=\rho_0(x)$. Suppose that $J_\rho (0) \neq 0$, then the following assertions hold.
\begin{enumerate}
\item The mapping $t \mapsto | J_\rho (t) |$ is nondecreasing, so that $y(t):=\frac{J_\rho (t)}{|J_\rho (t)|} \in \bbs^{d-1}$ is well defined, and there exists $y(t) \in \bbs^{d-1}$ and $y_\infty\in\bbs^{d-1}$ such that
\[
\lim_{t \to \infty} y(t) = y_{\infty}.
\]
\item There exists a unique $w \in \bbs^{d-1}$ such that the solution of 
\begin{align*}
\begin{cases}
\dot{x}=\mathbb{P}_{x^{\perp}}J_\rho (t)\\
x(0)=w
\end{cases}
\end{align*}
satisfies $\lim_{t \to \infty}x(t)=-y_{\infty}$. Furthermore, if $m$ is the mass of $\{w\}$ with respect to measure $\rho_0$, then
\[
0\leq m < \frac{1}{2} \quad \text{and} \quad \rho(t,\cdot) \to (1-m)\delta_{y_{\infty}}+m\delta_{-y_{\infty}} \text{ weakly as } t \to \infty. 
\] 
\end{enumerate}
\end{theorem}
\begin{proof}
For a proof, we refer to Theorem 1 in \cite{F-L}.
\end{proof}
\begin{remark}
Note that Theorem \ref{T2.4} is analogous to \eqref{M-2} in the following sense: if $\rho^0$ has no atoms(i.e. $m = 0$) and satisfies $J_{\rho_0} \neq 0$, then the measure $\rho$ converges weakly to a Dirac mass concentrated at $y_{\infty}$.
\end{remark}

\section{A global well-posedness} \label{sec:3}
\setcounter{equation}{0}
In this section, we study a global well-posedness of measure-valued solutions to \eqref{A-1} on $\mathbb{HS}^{d-1}$. Recall the Cauchy problem to kinetic LHS model becomes 
\begin{align}\label{C-1}
\begin{cases}
\partial_t f+\nabla_z \cdot(L[f]f)=0, \quad t>0, \quad (z, \Omega)\in\Xi, \\
\displaystyle L[f](z, \Omega)=\Omega z
 +\int_{\Xi}[\kappa_0(z_*-\langle{z_*, z}\rangle z)+\kappa_1(\langle{z, z_*}\rangle z-\langle{z_*, z}\rangle z)]f(t, z_*, \Omega_*)d\sigma_{z_*}d\Omega_*, \\
f(0,z, \Omega)=f_0(z, \Omega),
\end{cases}
\end{align}
where we use the notation $\nabla_z$ as a covariant derivative on $\mathbb{HS}^{d-1}$ with respect to $z$.

\subsection{Preparatory lemmas} \label{sec:3.1}
First, we study a canonical identification $\iota: \bbc^d\to \bbr^{2d}$ defined by
\[
z = (z^1,\cdots, z^d)\mapsto (\mathrm{Re}~z^1, \mathrm{Im}~z^1, \cdots, \mathrm{Re}~z^d, \mathrm{Im}~z^d).
\]
Then the restriction $\iota \big|_{\mathbb{HS}^{d-1}}$ is an one-to-one map between $\mathbb{HS}^{d-1}(\subseteq \bbc^d)$ and $\mathbb{S}^{2d-1}$.  First, we introduce an inner product on $\bbc^d$ as follows.
\begin{definition} \label{D3.1}
Let $z = (z^1, \cdots, z^d)$ and $w = (w^1, \cdots, w^d)$ be vectors in $\bbc^d$. Then, the inner product `` $\cdot$" between $z$ and $w$ is defined as follows:
\begin{equation} \label{C-1-1}
z \cdot w =\iota(z)\cdot \iota (w) = \sum_{i=1}^{d} \Big[  (\mathrm{Re} z^i) (\mathrm{Re} w^i )+  ( \mathrm{Im} z^i) ( \mathrm{Im} w^i)  \Big ],
\end{equation}
where dot product in $\iota(z)\cdot \iota (w) $ is the natural dot product in $\bbr^{2d}$. 
\end{definition}
Next, we study elementary properties of a dot product $``\cdot"$ between two complex vectors. 
\begin{lemma}\label{L3.1}
Let $z = (z^1, \cdots, z^d) , w = (w^1, \cdots, w^d)$ be vectors in $\bbc^d$ and $\beta, \gamma \in \bbc$. Then, one has 
\begin{enumerate}[(i)]
\item $(\mathrm{i}z) \cdot w = -  z \cdot(\mathrm{i}w)$,
\item $(\mathrm{i}z) \cdot( \mathrm{i}w) = z \cdot w$,
\item $(\beta w) \cdot z = (\mathrm{Re} \beta) w \cdot z  + (\mathrm{Im}\beta) (\mathrm{i}w) \cdot z$,
\item $(\beta w) \cdot (\gamma z) = (\beta \bar{\gamma} w) \cdot z$.
\end{enumerate}
\end{lemma}
\begin{proof}
\noindent (i)~Note that  
\[
\iota(\mathrm{i}z) = (-\mathrm{Im} z^1, \mathrm{Re} z^1, -\mathrm{Im} z^2, \mathrm{Re} z^2, \cdots, -\mathrm{Im}z^d, \mathrm{Re}z^d).
\]
By \eqref{C-1-1}, one has
\begin{align*}
(\mathrm{i}z) \cdot w &= \sum_{k = 1}^{d} \Big[ -(\mathrm{Im} z^k) (\mathrm{Re} w^k) + (\mathrm{Re} z^k) (\mathrm{Im} w^k) \Big] \\
&= - \sum_{\alpha = 1}^{d}  \Big[   (-\mathrm{Re} z^k) (\mathrm{Im} w^k) +  (\mathrm{Im} z^k) (\mathrm{Re} w^k) \Big] = -z \cdot (\mathrm{i}w).
\end{align*}
This yields the first relation. \newline

\noindent(ii)~We use the first identity to obtain
\[ (\mathrm{i}z) \cdot (\mathrm{i}w) =- z \cdot (\mathrm{i}(\mathrm{i}w) )=- z \cdot (-w) = z \cdot w.
\]
\noindent(iii) and (iv):~ By direct calculation, we have
\begin{equation} \label{C-1-2}
(\beta w) \cdot z = \Big ( (\mathrm{Re}\beta)w+(\mathrm{Im}\beta)(\mathrm{i}w) \Big) \cdot z  = (\mathrm{Re}\beta) (w \cdot z) + (\mathrm{Im}\beta) ((\mathrm{i}w) \cdot z),
\end{equation}
and
\begin{align*}
(\beta w) \cdot (\gamma z) &=\Big ( (\mathrm{Re}\beta)w+(\mathrm{Im}\beta)(\mathrm{i}w) \Big) \cdot \Big ( (\mathrm{Re}\gamma)z+(\mathrm{Im}\gamma)(\mathrm{i}z) \Big ) \\
&= (\mathrm{Re}\beta)(\mathrm{Re}\gamma) ( w \cdot  z)  +(\mathrm{Im}\beta)(\mathrm{Re}\gamma)  \Big( (\mathrm{i}w) \cdot z \Big) \\
&\hspace{0.2cm}+(\mathrm{Re}\beta)(\mathrm{Im}\gamma) ( w \cdot (\mathrm{i}z) ) +(\mathrm{Im}\beta)(\mathrm{Im}\gamma) \Big( (\mathrm{i}w) \cdot (\mathrm{i}z)  \Big).
\end{align*}
Now, we use the results of (i) and (ii) to obtain
\begin{align*}
( \beta w) \cdot (\gamma z) &=\big((\mathrm{Re}\beta)(\mathrm{Re}\gamma)+(\mathrm{Im}\beta)(\mathrm{Im}\gamma)\big) (w \cdot z) \\
&+\big((\mathrm{Im}\beta)(\mathrm{Re}\gamma)-(\mathrm{Re}\beta)(\mathrm{Im}\gamma)\big)  \Big( (\mathrm{i}w) \cdot z \Big) \\
&=\mathrm{Re}(\beta\bar{\gamma}) (w \cdot z) +\mathrm{Im}(\beta\bar{\gamma}) ( (\mathrm{i}w) \cdot z).
\end{align*}
Then, we use \eqref{C-1-2} to get the last identity:
\[ (\beta w) \cdot (\gamma z) = ( \beta\bar{\gamma}w) \cdot  z.
\]
\end{proof}
In what follows, we introduce definitions of the covariant derivative for scalar and vector-valued functions, respectively and divergence operator on $\mathbb{HS}^{d-1}$. To find the covariant derivative $\nabla_z$ on $\mathbb{HS}^{d-1}$ at the point $z\in \mathbb{HS}^{d-1}$, we use the covariant derivative $\tilde{\nabla}$ on the unit sphere $\bbs^{2d-1}$.
Let $f$ and $F$ be differentiable real-valued and vector-valued functions on $\mathbb{HS}^{d-1}$, respectively. Then, we define $\tilde{f}$ and ${\tilde F}$ on $\bbr^{2d}$ by 
\[   {\tilde f} := f \circ \iota^{-1}, \quad {\tilde F} := \iota \circ F \circ \iota^{-1}. \]
\begin{center}
\begin{tikzpicture}
  \matrix (m) [matrix of math nodes,row sep=3em,column sep=10em,minimum width=2em]
  {    \bbc^d& \bbr \\
     \bbr^{2d} & \\};
  \path[-stealth]
    (m-2-1) edge node [right] {$\iota^{-1}$} (m-1-1)
      (m-1-1)      edge [right] node [above] {$f$} (m-1-2)
    (m-2-1.east|-m-2-1) edge node [below] {\hspace{2cm}$\tilde{f}=f\circ\iota^{-1}$}
 (m-1-2);
\end{tikzpicture}\quad
\begin{tikzpicture}
  \matrix (m) [matrix of math nodes,row sep=3em,column sep=10em,minimum width=2em]
  {    \bbc^d& \bbc^d \\
     \bbr^{2d} & \bbr^{2d} \\};
  \path[-stealth]
    (m-2-1) edge node [right] {$\iota^{-1}$} (m-1-1)
      (m-1-1)      edge [right] node [above] {$F$} (m-1-2)
    (m-2-1.east|-m-2-2) edge node [below] {$\tilde{F}=\iota \circ F\circ \iota^{-1}$}
 (m-2-2)
    (m-1-2) edge node [left] {$\iota$} (m-2-2);
\end{tikzpicture}
\end{center}
\begin{definition} \label{D3.2}
Let $f$ and $F$ be a differentiable real-valued function and a vector field on $\mathbb{HS}^{d-1}$, respectively. 
\begin{enumerate}
\item
The covariant derivative of $f$ at $z \in \mathbb{HS}^{d-1}$ is defined as follows.
\begin{equation*} \label{C-1-3}
\nabla_zf(z):=\iota^{-1}(\tilde{\nabla}_{\iota(z)} \tilde{f}(\iota(z))).
\end{equation*}

\vspace{0.2cm}

\item
The divergence of $F$ at $z \in  \mathbb{HS}^{d-1}$ is defined as follows.
\begin{equation*} \label{C-1-4}
\nabla_z\cdot F(z):=\tilde{\nabla}_{\iota(z)}\cdot \tilde{F}(\iota(z)).
\end{equation*}
\item
The surface measure $d\sigma_z$ on $\mathbb{HS}^{d-1}$ is defined as a push-forward of the surface measure $d\sigma_\mathbf{x}$ on $\bbs^{2d-1}$ via  the inclusion map $\iota$:
\[
d\sigma_{z}:=\iota_*(d\sigma_\mathbf{x}).
\]
Then we have the following property:
\begin{equation*} \label{C-1-5}
\int_{\mathbb{HS}^{d-1}}f(z)d\sigma_z=\int_{\mathbb{S}^{2d-1}} f\circ \iota^{-1}(\mathbf{x})d\sigma_\mathbf{x},
\end{equation*}
where $d\sigma_\mathbf{x}$ is the usual surface area measure on the spherical surface $\mathbb{S}^{2d-1}$.
\end{enumerate}
\end{definition}
In the following lemma, we provide elementary identities for a later use. 
\begin{lemma}\label{L3.2}
The following assertions hold.
\begin{enumerate}
\item
Let $F$ be a vector field on $\mathbb{HS}^{d-1}$. Then one has
\[
\int_{\mathbb{HS}^{d-1}}\nabla_z\cdot F(z) d\sigma_z=0.
\]
\item
If $\phi(z)=z\cdot e$, then we have the following relation:
\[
\nabla_{z}\phi(z)=e-(e\cdot z)z.
\]
\end{enumerate}
\end{lemma}
\begin{proof}
(i)~By Definition \ref{D3.2} and Stokes' theorem, we have
\[
\int_{\mathbb{HS}^{d-1}}\nabla_z\cdot F(z) d\sigma_z =\int_{\mathbb{HS}^{d-1}}\tilde{\nabla}_{\iota(z)}\cdot \tilde{F}(\iota(z))d\sigma_z =\int_{\mathbb{S}^{2d-1}}\tilde{\nabla}_{\mathbf{x}}\cdot \tilde{F}(\mathbf{x})d\sigma_\mathbf{x}=0.
\]
\newline
(ii)~Recall that for a differentiable function $f\in C^1(\bbr^{2d})$, we have
\begin{equation} \label{C-2}
\tilde{\nabla}_\mathbf{x}f(\mathbf{x})=\mathbb{P}_{\mathbf{x}^\perp}( {\nabla}f\left(\mathbf{x})\right),\quad \text{for all}~\mathbf{x} \in \mathbb{S}^{2d-1},
\end{equation}
where $\nabla$ denotes a standard gradient on $\bbr^{2d}$, and $\mathbb{P}_{\mathbf{x}^\perp}$ is the projection onto the hyperplane $\{y\in\bbr^{2d}: y\cdot x=0\}$:
\begin{equation} \label{C-3}
\mathbb{P}_{\mathbf{x}^\perp}\mathbf{v}=\mathbf{v}-(\mathbf{v}\cdot\mathbf{x})\mathbf{x},\quad \text{for all} ~~\mathbf{v}\in\bbr^{2d},\quad \mathbf{x}\in\mathbb{S}^{2d-1}\subseteq \bbr^{2d}.
\end{equation}
It follows from \eqref{C-2} and \eqref{C-3} that 
\begin{equation*} \label{C-4}
\tilde{\nabla}_{\mathbf{x}}f(\mathbf{x})=\nabla f(\mathbf{x})-\left(\nabla f(\mathbf{x})\cdot \mathbf{x}\right) \mathbf{x}.
\end{equation*}
Thus, for any $g \in {C}^1(\bbc^{d}, \bbr)$, one has 
\begin{align*}
\begin{aligned} \label{C-4-1}
\nabla_z g(z)&=\iota^{-1}(\tilde{\nabla}_{\iota(z)}\tilde{g}(\iota(z))) =\iota^{-1}\left(\nabla \tilde{g}(\iota(z))-(\nabla \tilde{g}(\iota(z))\cdot \iota(z))\iota(z)\right) \\
&=\iota^{-1}\left(\nabla \tilde{g}(\iota(z))\right)-(\nabla \tilde{g}(\iota(z))\cdot \iota(z))z.
\end{aligned}
\end{align*}
Now, we substiutte $g(z)= \phi(z) = z\cdot e$ for a constant complex vector $e\in\bbc^{d}$ to get 
\[
\nabla_{\bbr^{2d}}\tilde{\phi}(z)=\nabla_{\bbr^{2d}}\iota(e)\cdot\iota(z)=\iota(e).
\]
This yields the desired estimate:
\begin{align*}
\nabla_{z}\phi(z)&=\iota^{-1}\left(\iota(e)\right)-(\iota(e)\cdot \iota(z))z
=e-(e\cdot z)z.
\end{align*}
\end{proof}
Next, we introduce the crucial lemma to derive a weak-solution formulation of \eqref{C-1}.
\begin{lemma}\label{L3.3}
Let $\phi$ and $F = (F^1, \cdots, F^d)$ be a smooth real-valued function and complex vector field on $\mathbb{HS}^{d-1}$, respectively. Then, the divergence of $\phi F$ satisfies
\[
\nabla_z\cdot \big(\phi(z)F(z)\big)=\phi(z)\nabla_z \cdot F(z)+\nabla_z\phi(z)\cdot F(z), \quad z \in \mathbb{HS}^{d-1}.
\]

\end{lemma}
\begin{proof} Recall that for $z \in \in \mathbb{HS}^{d-1}$, 
\[
F(z) := (F_1(z), \cdots, F_d(z)), \quad  (\phi F)(z) := \phi(z) F(z). 
\]
Now, we define $\widetilde{\phi F}$: 
\begin{align*}
\widetilde{\phi F}&: \iota(z)=(\mathrm{Re}~z^1, \mathrm{Im}~z^1, \cdots, \mathrm{Re}~z^d, \mathrm{Im}~z^d)\\
& \mapsto \big(\phi(\alpha)\mathrm{Re}~F^1(z), \phi(z)\mathrm{Im}~F^1(z), \cdots, \phi(z)\mathrm{Re}~F^d(z), \phi(z)\mathrm{Im}~F^d(z)\big) = \tilde{\phi}(\iota(z))\tilde F(\iota(z))
\end{align*}
so that 
\begin{align}\label{X-1}
\widetilde{(\phi F)}(\iota(z))= \tilde{\phi}(\iota(z))\tilde{F}(\iota(z)).
\end{align}
Therefore, one has
\begin{align*}
&\nabla_z\cdot \big(\phi(z)F(z)\big)
=\nabla_z\cdot \big((\phi F)(z)\big)
=\tilde{\nabla}_{\psi(z)}\cdot \big((\widetilde{\phi F})(\psi(z))\big)\\
&\hspace{0.5cm} =\tilde{\nabla}_{\psi(z)}\cdot \big(\tilde{\phi}(\psi(z))\tilde{F}(\iota(z))\big) &&(\mbox{by} \quad \eqref{X-1})\\
&\hspace{0.5cm} \overset{\mathrm{\star}}{=} \tilde{\nabla}_{\iota(z)}\tilde{\phi}(\iota(z)) \cdot \tilde{F}(\iota(z))
+ \tilde{\phi}(\psi(z)) \tilde{\nabla}_{\iota(z)} \cdot \tilde{F}(\iota(z))\\
&\hspace{0.5cm}= \iota^{-1}\Big(\tilde{\nabla}_{\iota(z)}\tilde{\phi}(\iota(z))\Big) \cdot \iota^{-1}\big(\tilde{F}(\iota(z))\big)
+ \tilde{\phi}(\iota(z)) \tilde{\nabla}_{\iota(z)} \cdot \tilde{F}(\iota(z))\\
&\hspace{0.5cm}= \iota^{-1}(\tilde{\nabla}_{\iota(z)}\tilde{\phi}(\iota(z))) \cdot F(z)
+ \phi(z)\big(\tilde{\nabla}_{\iota(z)}\cdot \tilde{F}(\iota(z))\big)\\
&\hspace{0.5cm}= \nabla_z\phi(z)\cdot F(z) + \phi(z)\nabla_z \cdot F(z),
\end{align*}
where the equality $\overset{\mathrm{\star}}{=}$ holds from the product rule for the divergence on $\mathbb{S}^{2d-1}$.
\end{proof}
Finally, we combine Lemma \ref{L3.2} and Lemma \ref{L3.3} to derive an integration by parts formula.
\begin{proposition}\label{P3.1}
Let $\phi$ and $F$ be a smooth real-valued function and a complex vector field on $\mathbb{HS}^{d-1}$, respectively. Then the following identity holds:
\[
\int_{\mathbb{HS}^{d-1}} \phi(z)\nabla_z \cdot F(z)d\sigma_z=-\int_{\mathbb{HS}^{d-1}}\nabla_z\phi(z)\cdot F(z)d\sigma_z.
\]
\end{proposition}
\begin{proof}
By Lemma \ref{L3.2} and Lemma \ref{L3.3}, one has 
\[
0
= \int_{\mathbb{HS}^{d-1}}\nabla_z\cdot \big(\phi(z)F(z)\big) d\sigma_z
= \int_{\mathbb{HS}^{d-1}}\left(\phi(z)\nabla_z \cdot F(z) + \nabla_z\phi(z)\cdot F(z) \right)d\sigma_z.
\]
This yields the desired identity.
\end{proof}
\subsection{A measure-theoretic formulation.} \label{sec:3.2}
In this subsection, we present a measure-theoretic formulation for the Cauchy problem \eqref{C-1}.  Let ${C}_w([0, T); \mathcal{P}(\Xi))$ be the set of all weakly continuous probability measure-valued function from $[0,T)$ to $\mathcal{P}(\Xi)$. \newline

Next, we recall a concept of measure-valued solution to the Cauchy problem \eqref{C-1} as follows. 

\begin{definition}\label{mvs_def}
 For $T\in[0, \infty]$, $\mu\in {C}_w([0, T); \mathcal{P}(\Xi))$ is a measure-valued solution to \eqref{C-1} with the initial measure $\mu_0\in\mathcal{P}(\Xi)$ if $\mu$ satisfies the following properties:
\begin{enumerate}
\item $\mu$ is weakly continuous:
\begin{equation} \label{C-5}
t\mapsto \langle{\mu_t, f}\rangle\mbox{ is continuous for all }  f\in C_0^1(\Xi).
\end{equation}
\item $\mu_t$ satisfies the following equation for all the test functions $\phi\in C_0^1([0, T) \times \Xi)$:
\begin{equation} \label{C-6}
\langle{\mu_t, \phi(t,\cdot, \cdot)}\rangle-\langle{\mu_0, \phi(0,\cdot, \cdot)}\rangle=\int_0^t\langle{\mu_s, \partial_s\phi+L[\mu_s]\cdot\nabla_z\phi}\rangle ds,
\end{equation}
where $L[\mu]$ is defined as
\begin{equation*}
L[\mu]( z, \Omega)=\Omega z+\int_{\Xi} \Big[\kappa_0(z_*-\langle{z_*, z}\rangle z)+\kappa_1(\langle{z, z_*}\rangle z-\langle{z_*, z}\rangle z) \Big ] d\mu(z_*, \Omega_*).
\end{equation*}
\end{enumerate}
\end{definition}
\begin{remark} \label{R3.1}
Note that an empirical measure made of particle solutions to \eqref{M-3} is in fact a measure-valued solution of \eqref{C-1} in the sense of Definition \ref{mvs_def}. More precisely, let $(z_i, \Omega_i)$ be a solution of system \eqref{M-3}. Then we define the empirical measure $\mu_t^N$ as follows:
\begin{equation} \label{C-7-1}
\mu_t^N:=\frac{1}{N}\sum_{j=1}^N\delta_{z_j(t)}\otimes\delta_{\Omega_j}.
\end{equation}
Note that for the empirical measure \eqref{C-7-1}, the alignment force $L[\mu_s^N]$ can be simplified as follows.
\begin{align}
\begin{aligned} \label{C-8}
L[\mu_s^N](z, \Omega)&=\Omega z+\int_{\Xi}[\kappa_0(z_*-\langle{z_*, z}\rangle z)+\kappa_1(\langle{z, z_*}\rangle z-\langle{z_*, z}\rangle z)]d\mu_s^N(z_*, \Omega_*)\\
&=\Omega z+\frac{1}{N}\sum_{k=1}^N[\kappa_0(z_k-\langle{z_k, z}\rangle z)+\kappa_1(\langle{z, z_k}\rangle z-\langle{z_k, z}\rangle z)].
\end{aligned}
\end{align}
Next, we check the defining relations \eqref{C-5} and \eqref{C-6} one by one. \newline

\noindent $\bullet$ (Verification of the relation \eqref{C-6}): We use \eqref{C-8} to see
\begin{align*}
&\int_{0}^t\langle{\mu^N_s, \partial_s \phi+L[\mu_s]\cdot\nabla_z\phi}\rangle ds\\
&=\int_{0}^t\Bigg\langle{\mu^N_s, \partial_s \phi+\left(\Omega z+\frac{1}{N}\sum_{k=1}^N[\kappa_0(z_k-\langle{z_k, x}\rangle z)+\kappa_1(\langle{z, z_k}\rangle z-\langle{z_k, z}\rangle z)]\right)\cdot\nabla_z\phi}\Bigg\rangle ds\\
&=\frac{1}{N}\sum_{i=1}^N\int_0^t\left(\partial_s\phi(s, z_i, \Omega_i)+\frac{1}{N}\sum_{k=1}^N[\kappa_0(z_k-\langle{z_k, z
_i}\rangle z_i)+\kappa_1(\langle{z_i, z_k}\rangle z-\langle{z_k, z_i}\rangle z_i)]\cdot\nabla_z\phi(s, z_i, \Omega_i)\right)ds\\
&=\frac{1}{N}\sum_{i=1}^N\int_0^t\left(\partial_s\phi(s,z_i, \Omega_i)+\frac{dz_i}{ds}\cdot\nabla_z\phi(s,z_i, \Omega_i)\right)ds =\frac{1}{N}\sum_{i=1}^N\int_0^t\frac{d}{ds}\phi(s,z_i, \Omega_i)ds\\
&=\frac{1}{N}\sum_{i=1}^N(\phi(t,z_i(t), \Omega)-\phi(0,z_i(0), \Omega))=\langle{\mu_t^N, \phi(t,\cdot, \cdot)}\rangle-\langle{\mu^N_0, \phi(0,\cdot, \cdot)}\rangle.
\end{align*}

\vspace{0.2cm}

\noindent $\bullet$ (Verification of the relation \eqref{C-5}):  Since $z_i(t)$ and $f$ are continuous, $\Omega_i$ is a constant matrix, and
\begin{align*}
\langle{\mu_t^N, f}\rangle=\frac{1}{N}\sum_{i=1}^Nf(z_i(t), \Omega_i),
\end{align*}
we can easily obtain that the map $ t\mapsto \langle{\mu_t^N, f}\rangle$ is continuous. Therefore, the empirical measure $\mu^N$ is a measure-valued solution of the Cauchy problem \eqref{C-1}.
\end{remark}

\subsection{Measure-valued solutions} \label{sec:3.3}
In this subsection, we present a uniform mean-field limit of the LHS model \eqref{M-3} and using this, we show that a global well-posedness of measure-valued solution to the Cauchy problem \eqref{C-1} for some class of initial data.

First, we provide $\ell^p$-stability estimates. For a state configuration $Z= \{z_i\}\in(\mathbb{HS}^{d-1})^N$, we define the $\ell^p$ norm of $Z$ as follows:
\[
\|Z\|_p := \left(\sum_{k=1}^N \|z_k\|^p\right)^\frac{1}{p}, \quad p \in [1, \infty),
\]
where $\|z\|$ is a standard $\ell^2$ norm of a complex vector $z\in\bbc^d$.
\vspace{0.2cm} 
\begin{proposition}[$\ell^p$-stability]\label{Lp_stability}
Let $Z=\{z_j\}$ and $\tilde{Z}=\{\tilde{z}_j\}$ be two solutions of system \eqref{M-3} with the initial data $Z^0=\{z_j^0\}$ and $\tilde{Z}^0=\{\tilde{z}_j^0\}$, respectively. Then, the following assertions hold. 
\begin{enumerate}
\item
For any fixed constant $T \in (0, \infty)$ and $p \in [1,\infty)$, there exists a time-depedent constant $G_T := \exp\big(2T(|\kappa_0|+|\kappa_0+2\kappa_1|)\big)>0$ which is independent of $N$ such that 
\[
\sup_{0\leq t\leq T}\|Z(t)-\tilde{Z}(t)\|_p \leq G_T\| Z(0)- \tilde{Z}(0) \|_p.
\]
\vspace{0.1cm}

\item Suppose $Z$ and $\tilde{Z}$ exhibit complete aggregation exponentially fast, i.e., there exist positive constants $A$ and $B$ such that
\[
\max \Big \{ ~\max_{i,j}\| z_i(t) - z_j(t) \|, ~ \max_{k,l}\| \tilde{z}_k(t) - \tilde{z}_l(t) \|~ \Big \}  < Ae^{-Bt}.
\]
Then, there exists a positive constant $G$ independent of $t$ such that 
\[
\sup_{0 \leq t < \infty} \|Z(t)-\tilde{Z}(t)\|_2 \leq G\| Z(0)- \tilde{Z}(0) \|_2, \quad t \in \bbr_+.
\]
Furthermore, if $A$ and $B$ are independent of $N$, then so is $G$.
\end{enumerate}
\end{proposition}
\begin{proof}
Since a proof is rather lengthy, we leave it to  Appendix \ref{App-B}.
\end{proof}
\begin{remark}\label{Lp_rmk}
By Theorem \ref{T2.3} and Proposition \ref{Lp_stability}, one can derive a uniform stability estimate for a homogeneous ensemble. More precisely, we assume that
\[ \Omega_j \equiv \Omega, \quad  |\kappa_1| < \frac{\kappa_0}{2}, \quad 
\max_{k,l}\left| 1 - \langle z^0_k, z^0_l \rangle \right| < 1 - \frac{2|\kappa_1|}{\kappa_0},
\]
Then, there exists a positive constant $G_\infty$ independent of $N$ such that 
\[
\sup_{0 \leq t < \infty}\|Z(t)-\tilde{Z}(t)\|_2 \leq G_\infty \| Z(0)- \tilde{Z}(0) \|_2.
\]
\end{remark}
Next, we provide a global well-posedness of measure-valued solution to the kinetic LHS model. First, we recall the concept of a mean-field limit which provides a construction of measure-valued solution to \eqref{C-1}.
\begin{definition}\label{D3.4}
We say the kinetic LHS model \eqref{C-1} is \emph{derivable}  from the LHS model \eqref{M-3} in $[0,T)$, if the following two properties hold.
\begin{enumerate}
\item For given initial measure $\mu_0 \in \mathcal{P}_p(\Xi)$, $\mu_0$ can be approximated by a sequence of empirical measures $\mu^N_0$ of \eqref{M-3} in Wasserstein metric:
\[
\lim_{N \to \infty} W_p(\mu_0,\mu^N_0) = 0.
\]
\item There exists a unique measure-valued solution $\mu$ of \eqref{C-1} with the initial data $\mu_0$, and for each $t \in [0,T)$, $\mu_t$ can be approximated by a sequence of empirical measures $\{\mu^N_t\}$ of \eqref{M-3} in the time interval $[0,T)$:
\[
\lim_{N \to \infty} \sup_{t \in [0,T)} W_p(\mu_t, \mu^N_t)=0.
\]
\end{enumerate}
\end{definition}
Next, we are ready to state our second main result on the unique solvability of \eqref{C-1}.
\begin{theorem}\label{T3.1} 
The following two assertions hold. 
\begin{enumerate}
\item  (Finite-in-time mean field limit):~For $T \in (0, \infty)$, the kinetic LHS model with identical natural frequency matrices $\Omega_j\equiv \Omega$ is derivable from the LHS model in a finite time interval $[0, T)$ in the sense of Definition \ref{D3.4}.

\vspace{0.1cm}

\item 
(Uniform-in-time mean filed limit):~Suppose system parameters and initial measure $\mu_0$ satisfy 
\[
|\kappa_1| < \frac{\kappa_0}{2},\quad 0\leq\sup_{z, w \in \mathrm{supp}(\mu_0)}|1-\langle z, w\rangle| < 1 - \frac{2|\kappa_1|}{\kappa_0}-\delta,\quad  \Omega_j\equiv \Omega,
\]
for some positive constant $\delta$. Then the kinetic LHS model is derivable from the LHS model in the whole time interval $[0, \infty)$ with respect to $W_2$-metric.
\end{enumerate} 
\end{theorem}
\begin{proof} 
(1)~(Proof of the first statement): We modify the proof of Theorem 3 of \cite{HKLN2}. First, we construct a sequence of empirical measures converging to a measure-valued solution for \eqref{C-1} in Wasserstein metric, and then, 
we verify the validity and uniqueness of the measure-valued solution. Since a proof is rather lengthy, we divide its proof into two parts.\\
 
\noindent$\bullet$ \textbf{Part A} (Construction of approximate solutions): First, we recall the result of Theorem 6.18 in \cite{Vi}. For a Polish space $(X,d)$, let $\mathcal{P}(X)$ be a space of probability measures on $X$, which can be equipped with Wasserstein metric from Remark \ref{Wp_remark}. For any given $p \in [1,\infty)$ and a probability measure $\mu \in \mathcal{P}(X)$, the set
\begin{align*}
\left\{ \sum_{i \in I} a_i \delta_{z_i} ~ : ~ \forall n \in \mathbb{N}, \quad 0 \leq a_n \in \mathbb{Q}, \quad \sum_{i \in I} a_i=1, \quad I \textrm{ is a finite subset of } \mathbb{N} \right\}
\end{align*}
is a dense subset of $\mathcal{P}_p(X)$.
Therefore, we can approximate  the initial measure $\mu_0$ by a sequence of empirical measures $\{ \mu_0^N \}$:
\begin{align*}
\lim_{N \to \infty} W_p(\mu_0^N, \mu_0) = 0,
\end{align*}
where $\mu_0^N$ is the sum of $N$ suitable Dirac measures uniformly weighted by $\frac{1}{N}$. To approximate the quantity $W^p_p(\mu_0^n,\mu_0^m)$, we denote
\begin{align*}
\mu_0^n = \frac{1}{n}\sum_{i=1}^n \delta_{z_{i_0}}, \quad \mu_0^m = \frac{1}{m}\sum_{j=1}^m \delta_{\bar{z}_{j_0}}.
\end{align*}
Since above measures are concentrated on the finite number of particles, infimum in the definition of $W_p(\mu_0^n,\mu_0^m)$ is achieved and therefore, one has 
\begin{align*}
W^p_p(\mu_0^n,\mu_0^m) = \frac{1}{nm} \sum_{i=1}^n \sum_{j=1}^m a_{ij} \| z_{i_0} - \bar{z}_{j_0} \|_p^p,
\end{align*}
for some optimal plan $(a_{ij})$ satisfying
\begin{align*}
\sum_{k=1}^n a_{kj}=n, \quad \sum_{l=1}^m a_{il}=m, \quad 0 \leq a_{ij} \in \bbr, \quad i,j \in \mathcal{N}.
\end{align*}
By perturbing $a_{ij}$, we can also approximate optimal strategy by rational strategy:
\begin{align}\label{rational}
\sum_{k=1}^n r_{kj}=n, \quad \sum_{l=1}^m r_{il}=m, \quad 0 \leq r_{ij} \in \mathbb{Q}, \quad i,j \in \mathcal{N},
\end{align}
in the sense that
\begin{align*}
0 \leq \frac{1}{nm} \sum_{i=1}^n \sum_{j=1}^m r_{ij} \| z_{i_0} - \bar{z}_{j_0} \|_p^p - W^p_p(\mu_0^n,\mu_0^m)  < \varepsilon,
\end{align*}
for arbitrary $\varepsilon>0$. Next, we introduce a common denominator $D$
\begin{align*}
r_{ij} := \frac{N_{ij}}{D}, \quad N_{ij},D \in \bbz_+,
\end{align*}
to rewrite rational strategy:
\begin{align}
\begin{aligned}\label{rational3}
\frac{1}{nm} \sum_{i=1}^n \sum_{j=1}^m r_{ij} \| z_{i_0} - \bar{z}_{j_0} \|_p^p =\frac{1}{nmD} \sum_{i=1}^n \sum_{j=1}^m N_{ij} \| z_{i_0} - \bar{z}_{j_0} \|_p^p  =\frac{1}{nmD} \sum_{k=1}^{mnD} \| z_{k_0} - \bar{z}_{k_0} \|_p^p,
\end{aligned}
\end{align}
where $\| z_{i_0} - \bar{z}_{j_0} \|_p^p$ is counted $N_{ij}$ times when $k$ runs through $1$ to $mnD$ in the last term.

On the other hand, we can associate rational numbers $(r_{ij})$ satisfying \eqref{rational} to a transport plan with marginals $\mu_t^n$ and $\mu_t^m$. Therefore, by the same procedure as in \eqref{rational3}, we have
\begin{align*}
W_p^p(\mu_t^n,\mu_t^m) \leq  \frac{1}{nmD} \sum_{k=1}^{mnD} \| z_{k}(t) - \bar{z}_{k}(t) \|_p^p.
\end{align*}
Let $T \in (0,\infty)$ be given.  Then we use Proposition \ref{Lp_stability} and obtain a constant $C=\max{(G_T,1)}$ to see
\begin{align}
\begin{aligned}\label{cor_pf2}
W_p^p(\mu_t^n,\mu_t^m) &\leq \frac{1}{nmD} \sum_{k=1}^{mnD} \| z_{k}(t) - \bar{z}_{k}(t) \|_p^p \\
&\leq \frac{C^p}{nmD} \sum_{k=1}^{mnD} \| z_{k_0} - \bar{z}_{k_0} \|_p^p \leq C^p\bigg(W_p^p(\mu_0^n, \mu_0^m)+\varepsilon\bigg),
\end{aligned}
\end{align}
for any $t \in (0,T]$, where $\varepsilon \to 0$ as $(r_{ij}) \to (a_{ij})$. \newline

Since $W_p^p(\mu_0^n, \mu_0^m)$ can be taken arbitrarily small by enlarging $n$ and $m$, we can conclude that $W_p(\mu_t^n,\mu_t^m)$ is a Cauchy sequence. From the completeness of $\mathcal{P}_p(\Xi)$(see Theorem 6.18 of \cite{Vi}), we can define a measure-valued solution $\mu$ with the initial data $\mu_0$, as a weak limit of $\mu_t^N$ in the Wasserstein metric.\\

\noindent$\diamond$ \textbf{Part B} (weak limit $\mu$ is a unique measure-valued solution): First, we show that the defining relation \ref{C-5} of Definition \ref{mvs_def} is satisfied (see \cite{H-L} for the detailed arguments). To check \eqref{C-6} in Definition \ref{mvs_def}, we need to show that $\mu$ satisfies the following equation for all the test functions $\phi\in C_0^1([0, T) \times \Xi)$:
\[
\langle{\mu_t, \phi(t,\cdot, \cdot)}\rangle-\langle{\mu_0, \phi(0,\cdot, \cdot)}\rangle=\int_0^t\langle{\mu_s, \partial_s\phi+L[\mu_s]\cdot\nabla_z\phi}\rangle ds,
\]
where $L[\mu]$ is defined by
\[
L[\mu](t,z, \Omega)=\Omega z+\int_{\Xi}[\kappa_0(z_*-\langle{z_*, z}\rangle z)+\kappa_1(\langle{z, z_*}\rangle z-\langle{z_*, z}\rangle z)]d\mu(z_*, \Omega_*).
\]
Since $\mu^N$ is a measure-valued solution (see Remark \ref{R3.1}), one has 
\begin{equation} \label{C-9}
\langle{\mu^N_t, \phi(t,\cdot, \cdot)}\rangle-\langle{\mu^N_0, \phi(0,\cdot, \cdot)}\rangle=\int_0^t\langle{\mu^N_s, \partial_s\phi+L[\mu^N_s]\cdot\nabla_z\phi}\rangle ds.
\end{equation}

\vspace{0.2cm}

\noindent Due to Remark \ref{Wp_remark}, L.H.S. of \eqref{C-9} becomes 
\[
\langle{\mu_t^N, \phi(t,\cdot, \cdot)}\rangle-\langle{\mu_0^N, \phi(0,\cdot, \cdot)}\rangle \to \langle{\mu_t, \phi(t,\cdot, \cdot)}\rangle-\langle{\mu_0, \phi(0,\cdot, \cdot)}\rangle \quad \text{as} \quad N \to \infty.
\]
\noindent For R.H.S. of \eqref{C-9}, we claim:
\begin{equation} \label{C-10}
\langle{\mu_s^N, \partial_s\phi+L[\mu_s^N]\cdot\nabla_z\phi}\rangle
\to \langle{\mu_s, \partial_s\phi+L[\mu_s]\cdot\nabla_z\phi}\rangle \quad \text{uniformly with respect to } s.
\end{equation}
{\it Proof of \eqref{C-10}}: Again from Remark \ref{Wp_remark}, we have 
\begin{align*}
\left| \langle{\mu_s^N, \partial_s\phi}\rangle - \langle{\mu_s, \partial_s\phi}\rangle \right| \to 0 \quad \text{as} \quad N \to \infty.
\end{align*}
Therefore, it is enough to make the following term arbitrary small:
\begin{align*}
&\left| \langle{\mu_s^N, L[\mu_s^N]\cdot\nabla_z\phi}\rangle - \langle{\mu_s, L[\mu_s]\cdot\nabla_z\phi}\rangle \right|.
\end{align*}
This can be obtained from the following estimate:
\begin{align*}
&\left| \langle{\mu_s^N, L[\mu_s^N]\cdot\nabla_z\phi}\rangle - \langle{\mu_s, L[\mu_s]\cdot\nabla_z\phi}\rangle \right| \\
&\hspace{0.5cm} =  \left| \langle{\mu_s^N-\mu_s, L[\mu_s]\cdot\nabla_z\phi}\rangle + \langle{\mu_s^N, \left(L[\mu_s^N]-L[\mu_s]\right)\cdot\nabla_z\phi}\rangle \right|\\
&\hspace{0.5cm}\leq \left| \langle{\mu_s^N-\mu_s, L[\mu_s]\cdot\nabla_z\phi}\rangle \right|+\sup \left( \left(L[\mu_s^N]-L[\mu_s]\right) \cdot\nabla_z\phi\right).
\end{align*}
Recall that $\mu_s^N$ converges weakly to $\mu_s$ and this convergence is uniform from the fact that the estimate \eqref{cor_pf2} is time-invariant.  Therefore, the estimate \eqref{C-10} is achieved and this implies 
\[
\int_0^t\langle{\mu_s^N, \partial_s\phi+L[\mu_s^N]\cdot\nabla_z\phi}\rangle ds
\to \int_0^t\langle{\mu_s, \partial_s\phi+L[\mu_s]\cdot\nabla_z\phi}\rangle ds,
\]
for any test function $\phi\in C_0^1([0, T) \times \Xi)$.  Therefore we have an existence of measure-valued solution. The uniqueness can be followed by the same argument as in \cite{F-L}. \vspace{0.5cm}
 \\
(2) (Proof of the second statement):~Let $\mu^N_0=\frac{1}{N}\sum_{i=1}^N\delta_{z_i^0}$ be the initial empirical measure and $\mu_0$ be the given initial measure. Since $\mathrm{supp}(\mu_0)$ is compact, we may restrict the domain $\mathbb{HS}^{d-1}$ to $\mathrm{supp}(\mu^0)$, so that support $\mu^N_0$ is a subset of $\mu_0$(See \cite{Vi} for the detailed argument). Therefore, by a priori condition on $\mathrm{supp}(\mu_0)$ and Remark \ref{Lp_rmk}, the estimation \eqref{cor_pf2} is valid uniformly for any $t \in (0, \infty]$.
\end{proof}

\vspace{0.5cm}

As a corollary of Theorem \ref{T3.1}, one obtains a finite-time stability estimate for measure-valued solutions with respect to initial measures. 

\begin{corollary}
\emph{(Finite-in-time stability)} \label{C3.1} For finite $T>0$, let $\mu, \nu\in C_{w}(\bbr_+;\mathcal{P}(\Xi))$ be the measure-valued solutions to the kinetic LHS model with the initial measures $\mu_0, \nu_0\in\mathcal{P}_p(\Xi)$, respectively. Then, there exists a positive constant $C = C(T)$ such that
\[
W_p(\mu_t, \nu_t)\leq C(T) \cdot W_p(\mu_0, \nu_0), \quad t \in [0,T).
\] 
\end{corollary}
\begin{proof}
This is a direct consequence of \eqref{cor_pf2} and the triangle inequality.
\end{proof}

\begin{remark}\label{Fin_rmk}
For $p =2$, we get an uniform-in-time stabiltiy of $N$-particle system in Proposition \ref{Lp_stability}. It follows that under the same assumption with Corollary \ref{C3.1} except for the finite-time condition, we have 
\[
W_2(\mu_t, \nu_t) \leq C \cdot W_2(\mu_0, \nu_0), \quad t \in [0, \infty),
\]
where $C$ is indepentent of $t$.
\end{remark}


\section{Emergent behaviors of the kinetic LHS model} \label{sec:4}
\setcounter{equation}{0}
In this section, we study emergent behaviors of the kinetic LHS model. In \cite{G-H}, emergent dynamics of the mean-field kinetic model for the LM model has been investigated. Although the LM model and LS model are different, we basically follow the same strategy in \cite{G-H} to analyze emergent dynamics of the kinetic LHS model for a homogeneous ensemble. Consider the LHS model \eqref{M-3} with the same natural frequency matrix:
\[ \Omega_j\equiv \Omega \quad \mbox{for all $j\in\mathcal{N}$}. \]
First, we define the notion of complete aggregation in a kinetic setting. Let $\pi$ be the projection map $\pi : \mathbb{HS}^{d-1}\times \mathrm{Skew}_{d}\mathbb{C} \rightarrow \mathbb{HS}^{d-1}$. Then we define a distribution function $\rho$ by the push-forward of $f$ with respect to $\pi$:
\[
\rho (t,z) = \rho_t(z) := \pi \# f_t(z, \Omega),
\]
where $f_t(z, \Omega)=f(t,z, \Omega)$. For any measurable set $A \subseteq\Xi$, we have 
\[
\rho(t, A) = f(\pi^{-1}(t, A)).
\]
Now we can write the aggregation force in terms of $\rho$ : 
\begin{align*}
\begin{aligned}
L[f](z, \Omega)&=\Omega z+\int_{\Xi}[\kappa_0(z_*-\langle{z_*, z}\rangle z)+\kappa_1(\langle{z, z_*}\rangle z-\langle{z_*, z}\rangle z)]f(t,z_*, \Omega_*)d\sigma_{z_*}d\Omega_*\\
&=\Omega z+\int_{\mathbb{HS}^{d-1}}[\kappa_0(z_*-\langle{z_*, z}\rangle z)+\kappa_1(\langle{z, z_*}\rangle z-\langle{z_*, z}\rangle z)]\rho(t,z_*)d\sigma_{z_*}.
\end{aligned}
\end{align*}
Then, equation \eqref{C-1} can be rewritten in terms of $\rho$ as  follows:
\begin{align}\label{D-1}
\begin{cases}
\partial_t \rho+\nabla_z \cdot(L[\rho]\rho)=0, \quad t > 0, \quad z\in\mathbb{HS}^{d-1}, \\
\displaystyle L[\rho](z)=\Omega z+\int_{\mathbb{HS}^{d-1}}[\kappa_0(z_*-\langle{z_*, z}\rangle z)+\kappa_1(\langle{z, z_*}\rangle z-\langle{z_*, z}\rangle z)]\rho(t,z_*)d\sigma_{z_*}\\
\rho(0,z)=\rho_0(z),
\end{cases}
\end{align} 
where we abused the notation $L[\rho](z):=L[f](z, \Omega)$. Note that for the particle model \eqref{M-3}, we have the following equivalent formulation of complete aggregation \eqref{M-2}:
\begin{align}
\begin{aligned} \label{D-1-1}
&\lim_{t \to \infty}\max_{i,j}\|z_i(t)-z_j(t)\|=0  \\
& \hspace{0.5cm} \Longleftrightarrow \quad \lim_{t \to \infty}\max_{k}\|z_k(t)-z_c(t)\|=0 \quad \Longleftrightarrow \quad \lim_{t \to \infty}\max_{k}W_p(\delta_{z_k(t)},\delta_{z_c(t)})=0,
\end{aligned}
\end{align}
where the second arrow follows from $W_p(\delta_x,\delta_y)=\|x-y\|$. \newline

The last formulation in \eqref{D-1-1} can be interpreted in terms of measures. This motivates the following concept of complete aggregation for the kinetic LHS model  \eqref{D-1}.
\begin{definition}\label{measure_cg}
\emph{(Complete aggregation)} The kinetic LHS model \eqref{D-1} exhibits complete aggregation, if for any measure-valued solution $\rho = \rho(t,z)$ of \eqref{D-1}, there exists a time-dependent Dirac measure $\delta_{w(t)}$ such that 
\[
\lim_{t\to\infty}W_2(\rho(t, \cdot), \delta_{w(t)}) = 0.
\]
\end{definition} 
Our third main result is the following result on emergent dynamics.
\begin{theorem} \label{T4.1}
Suppose coupling strengths satisfy
\[
\kappa_0>0,\quad \kappa_0+2\kappa_1\geq0,
\]
and let $\rho$ be a solution to \eqref{D-3}.Then, we have 
\[
\lim_{t\to\infty} \int_{\mathbb{HS}^{d-1}} \Big(\|J_\rho\|^2-|z\cdot J_\rho|^2 \Big )d\sigma_z=0,
\]
where $J_\rho$ is the expectation of $z$:
\begin{equation} \label{D-1-2}
J_\rho = \int_{\mathbb{HS}^{d-1}}z\rho(t,z)d\sigma_z.
\end{equation}
 \end{theorem}
 \begin{proof} We leave its proof in Section \ref{sec:4.3}.
  \end{proof}

\subsection{Solution splitting property}  \label{sec:4.1}
In this subsection, we study a solution splitting property for \eqref{D-1}.  Consider the continuity equation \eqref{D-1}:
\begin{align}\label{D-2}
\begin{cases}
\partial_t \tilde{\rho}+\nabla_z \cdot(\tilde{L}[\tilde{\rho}]\tilde{\rho})=0, \quad t > 0, \quad z\in\mathbb{HS}^{d-1}, \\
\displaystyle\tilde{L}[\tilde{\rho}](z)=\int_{\mathbb{HS}^{d-1}}[\kappa_0(z_*-\langle{z_*, z}\rangle z)+\kappa_1(\langle{z, z_*}\rangle z-\langle{z_*, z}\rangle z)]\tilde{\rho}(t,z_*)d\sigma_{z_*},\\
\tilde{\rho}(0,z)=\tilde{\rho}_0(z).
\end{cases}
\end{align} 
In the following proposition, we clarify the solution splitting property of \eqref{D-1} more precisely. 
\begin{proposition}  \label{P4.1}
\emph{(Solution splitting property)}
Let $\rho$ be a smooth solution to \eqref{D-1}. If we set
\begin{equation} \label{D-2-1}
\tilde{\rho}(t,z)=\rho(t,e^{\Omega t}z),
\end{equation}
then $\tilde{\rho}$ satisfies \eqref{D-2}.
\end{proposition}
\begin{proof}
It follows from \eqref{D-2-1} that 
\[
\rho(t,z)=\tilde{\rho}(t, e^{-\Omega t}z).
\]
Recall that $\rho$ satisfies
\begin{align*}
\begin{cases}
\displaystyle \partial_t \rho+\nabla_z \cdot(L[\rho]\rho)=0, \quad t>0, \quad z \in \mathbb{HS}^{d-1}, \\
\displaystyle L[\rho](z)=\Omega z+\int_{\mathbb{HS}^{d-1}}[\kappa_0(z_*-\langle{z_*, z}\rangle z)+\kappa_1(\langle{z, z_*}\rangle z-\langle{z_*, z}\rangle z)]\rho(t,z_*)d\sigma_{z_*}.
\end{cases}
\end{align*}
Then, we use \eqref{D-2-1} to see
\begin{align}
\begin{aligned} \label{D-2-2}
\partial_t \tilde{\rho}(t, w) &= \frac{d}{dt}\left( \rho(t,e^{\Omega t}w)\right)=\left(\Omega e^{\Omega t}w\right)\cdot \nabla_z \rho(t,e^{\Omega t}w)+\partial_t\rho(t,e^{\Omega t}w), \\
\tilde{L}[\tilde{\rho}](w) &=\int_{\mathbb{HS}^{d-1}}[\kappa_0(w_*-\langle{w_*, w}\rangle w)+\kappa_1(\langle{w, w_*}\rangle w-\langle{w_*, w}\rangle w)]\tilde{\rho}(t,w_*)d\sigma_{w_*}\\
&=\int_{\mathbb{HS}^{d-1}}[\kappa_0(w_*-\langle{w_*, w}\rangle w)+\kappa_1(\langle{w, w_*}\rangle w-\langle{w_*, w}\rangle w)]\rho(t,e^{\Omega t}w_*)d\sigma_{w_*}\\
&=e^{-\Omega t}\int_{\mathbb{HS}^{d-1}} \Big[\kappa_0(e^{\Omega t}w_*-\langle{e^{\Omega t}w_*, e^{\Omega t}w}\rangle e^{\Omega t}w) \\
& \hspace{0.2cm} +\kappa_1(\langle{e^{\Omega t}w, e^{\Omega t}w_*}\rangle e^{\Omega t}w-\langle{e^{\Omega t}w_*, e^{\Omega t}w}\rangle e^{\Omega t}w) \Big ]\rho(t,e^{\Omega t}w_*)d\sigma_{w_*}.
\end{aligned}
\end{align}
Now, we use \eqref{D-2-2} and the fact that  $e^{\Omega t}$ is a rotational operator which preserves the surface area  (i.e. roughly $d\sigma_{w_*}=dS_{e^{\Omega t} w_*}$)  to obtain
\begin{align*}
e^{\Omega t}\tilde{L}[\tilde{\rho}](e^{-\Omega t}w)&=\int_{\mathbb{HS}^{d-1}}[\kappa_0(w_*-\langle{w_*, w}\rangle w)+\kappa_1(\langle{w, w_*}\rangle w-\langle{w_*, w}\rangle w)]\rho(t,w_*)d\sigma_{w_*},\\
&=L[\rho](t,w)-\Omega w
\end{align*}
to see
\[ \tilde{L}[\tilde{\rho}](w)=e^{-\Omega t}L[\rho](t,e^{\Omega t} w)-\Omega w. \]
This implies 
\[ \tilde{L}[\tilde{\rho}]\tilde{\rho}(t,w)=e^{-\Omega t}L[\rho]\rho(t,e^{\Omega t} w)- \rho(t,e^{\Omega t} w)\Omega w. \]
Now we have
\begin{align} 
\begin{aligned} \label{D-2-3}
\nabla_z\cdot (\tilde{L}[\tilde{\rho}]\tilde{\rho})(t,w) &=\nabla_z\cdot\left(L[\rho]\rho\right)(t,e^{\Omega t} w) \\
&-(\Omega e^{\Omega t}w)\cdot \nabla_z\rho(t, e^{\Omega t} w)-\rho(t,e^{\Omega t}w)\left(\nabla_w\cdot \Omega w\right).
\end{aligned}
\end{align}
From $w=x+\mathrm{i}y$ and simple calculations, it follows that
\begin{align}
\begin{aligned} \label{D-2-4}
\nabla_w\cdot \Omega w&=\nabla_x\cdot \mathrm{Re}(\Omega w)+\nabla_y\cdot \mathrm{Im}(\Omega w)\\
&=\frac{\partial}{\partial x_\alpha}\mathrm{Re}\left(\Omega_{\alpha\beta}w_\beta\right)+
\frac{\partial}{\partial y_\alpha}\mathrm{Im}\left(\Omega_{\alpha\beta}w_\beta\right)
\\
&=\frac{\partial}{\partial x_\alpha}\mathrm{Re}\left(\Omega_{\alpha\beta}x_\beta\right)+
\frac{\partial}{\partial y_\alpha}\mathrm{Im}\left(\mathrm{i}\Omega_{\alpha\beta}y_\beta\right)\\
&=\frac{\partial}{\partial x_\alpha}\mathrm{Re}\left(\Omega_{\alpha\beta}x_\beta\right)+
\frac{\partial}{\partial y_\alpha}\mathrm{Re}\left(\Omega_{\alpha\beta}y_\beta\right)\\
&=\mathrm{Re}(\Omega_{\alpha\beta}\delta_{\alpha\beta})+\mathrm{Re}(\Omega_{\alpha\beta}\delta_{\alpha\beta})=2\mathrm{Re}(\Omega_{\alpha\alpha})=2\mathrm{tr}(\Omega)=0.
\end{aligned}
\end{align}
Here we used the fact that $\Omega$ is skew-Hermitian. \newline

If we combine \eqref{D-2-3} and \eqref{D-2-4}, we have
\begin{align*}
\nabla_z\cdot (\tilde{L}[\tilde{\rho}]\tilde{\rho})(t,w)=\nabla_z\cdot\left(L[\rho]\rho\right)(t,e^{\Omega t} w)-(\Omega e^{\Omega t}w)\cdot \nabla_z\rho(t,e^{\Omega t} w).
\end{align*}
Finally, we have the following relation:
\begin{align*}
&\partial_t \tilde{\rho}(t,w)+\nabla_z \cdot(\tilde{L}[\tilde{\rho}]\tilde{\rho})(t,w)\\
& \hspace{0.5cm} =\left(\Omega e^{\Omega t}w\right)\cdot \nabla_z \rho(t,e^{\Omega t}w)+\partial_t\rho(t,e^{\Omega t}w) \\
& \hspace{0.5cm} +\nabla_z\cdot\left(L[\rho]\rho\right)(t,e^{\Omega t} w)-(\Omega e^{\Omega t}w)\cdot \nabla_z\rho(t,e^{\Omega t} w)\\
& \hspace{0.5cm} =\partial_t\rho(t,e^{\Omega t}w)+\nabla_z\cdot\left(L[\rho]\rho\right)(t,e^{\Omega t} w)=0.
\end{align*}
In the last equality, we used the fact that $\rho$ is a solution of \eqref{D-1}.
\end{proof}
\subsection{Order parameter} \label{sec:4.2}
Recall the definition of the order parameter $R$:
\begin{equation} \label{D-2-5}
R(t) := | J_\rho | = \left|\int_{\Xi}z f(t,z, \Omega) d\sigma_zd\Omega\right|=\left|\int_{\mathbb{HS}^{d-1}}z\rho(t,z)d\sigma_z\right|.
\end{equation}
In the rest of this subsection, for notational simplicity, we suppress $t$ dependence in $R, f$ and $\rho$:
\[ R := R(t), \quad f(z) := f(t,z), \quad \rho(z) := \rho(t,z). \]
Then, we see that 
\begin{align*}
R^2 &= \left|\int_{\mathbb{HS}^{d-1}}z\rho(z)d\sigma_z\right|^2
=\left\langle \int_{\mathbb{HS}^{d-1}}z_1\rho(z_1)d\sigma_{z_1} , \int_{\mathbb{HS}^{d-1}}z_2\rho(z_2)d\sigma_{z_2}\right\rangle\\
&=\iint_{(\mathbb{HS}^{d-1})^2}\langle z_1, z_2\rangle \rho(z_1)\rho(z_2)d\sigma_{z_1}d\sigma_{z_2}.
\end{align*}
In what follows, we will show 
\[ \frac{dR^2}{dt} \geq 0 \quad \mbox{and} \quad \sup_{0 \leq t < \infty} \Big| \frac{d^2 R^2}{dt^2}  \Big | < \infty. \]
Then, these estimates yield the desired estimate in Theorem \ref{T4.1}. Thanks to the solution splitting property in Proposition \ref{P4.1},  we may assume $\Omega=0$ without loss of generality.  Thus, $\rho$ satisfies  \begin{align}\label{D-3}
\begin{cases}
\partial_t \rho+\nabla_z \cdot(\rho L[\rho])=0, \quad z \in \mathbb{HS}^{d-1}, \quad t>0,\\
L[\rho](z)=\displaystyle\int_{ \mathbb{HS}^{d-1}}[\kappa_0(z_*-\langle{z_*, z}\rangle z)+\kappa_1(\langle{z, z_*}\rangle z-\langle{z_*, z}\rangle z)]\rho(z_*)d\sigma_{z_*}.
\end{cases}
\end{align}
In what follows, we present three elementary lemmas for the proof of Theorem \ref{T4.1}. 
\begin{lemma}\label{L4.1}
Let $z$ and $w$ be complex vectors in $\mathbb{C}^{d}$. Then, the functional introduced in \eqref{A-3} satisfies
\[
\langle z, w \rangle = z \cdot w - \mathrm{i} (z \cdot (\mathrm{i}w)),
\]
where $\mathrm{i}z := (\mathrm{i} z^1,\cdots, \mathrm{i} z^d) \in \bbc^d$.
\end{lemma} 

\begin{proof}
By definitions of \eqref{A-3}, one has
\begin{align*}
\langle z, w\rangle&=\sum_{k=1}^d \overline{z^k} w^k =\sum_{k=1}^d\left(\mathrm{Re} z^k -\mathrm{i}\mathrm{Im} z^k \right)\left(\mathrm{Re} w^k+\mathrm{i}\mathrm{Im} w^k \right)\\
&=\sum_{k=1}^d\left(\mathrm{Re} z^k \mathrm{Re} w^k + \mathrm{Im} z^k \mathrm{Im}w^k \right)
+\mathrm{i}\sum_{k=1}^d\left(\mathrm{Re} z^k \mathrm{Im} w^k -\mathrm{Im} z^k \mathrm{Re} w^k \right)\\
&=\sum_{k=1}^d\left(\mathrm{Re} z^k \mathrm{Re} w^k + \mathrm{Im} z^k \mathrm{Im}w^k \right)
-\mathrm{i}\sum_{k=1}^d\left(\mathrm{Re} z^k \mathrm{Re} (\mathrm{i}w^k) +\mathrm{Im} z^k \mathrm{Im} (\mathrm{i} w^k) \right)\\
&=z\cdot w-\mathrm{i}(z\cdot (\mathrm{i}w)).
\end{align*} 
\end{proof}
\begin{definition}  \label{D4.2}
For $z \in \mathbb{HS}^{d-1}\subset\mathbb{C}^d$ and $v \in \bbc^d$, we define three $\bbc^d$-valued maps on $\bbc^d$ as follows.
\begin{enumerate}
\item Define a map $Q_z : \mathbb{C}^{d} \rightarrow \mathbb{C}^{d}$ by 
\[
Q_z(v) := \kappa_0(v - \langle v, z \rangle z) + \kappa_1(\langle z, v \rangle - \langle v, z \rangle)z,\quad \forall~v\in\bbc^d.
\]  
\item Define linear operators $\mathbb{P}_z$ and $\mathbb{P}_{z^\perp}: \mathbb{C}^{d} \to \mathbb{C}^{d} $ by
\[
\mathbb{P}_z(v) := (z \cdot v) z, \quad \mathbb{P}_{z^\perp}(v) := v - (z \cdot v)z.
\] 
\end{enumerate}
\end{definition}
In the following lemma, we study basic properties of the maps introduced in Definition \ref{D4.2}. 
\begin{lemma} \label{L4.2}
For $z \in \mathbb{HS}^{d-1}\subset\mathbb{C}^d$ and $v \in \bbc^d$, one has 
\begin{eqnarray*}
&& (i)~\mathbb{P}_z^2=\mathbb{P}_z\quad \mathbb{P}_{z^\perp}^2=\mathbb{P}_{z^\perp}. \\
&& (ii)~\mathbb{P}_z=\iota^{-1}\circ \mathbb{P}_{\iota(z)}\circ \iota,\quad \mathbb{P}_{z^\perp}=\iota^{-1}\circ \mathbb{P}_{\iota(z)^\perp}\circ \iota. \\
&& (iii)~ L[\rho] = Q_z(J_\rho), \quad  Q_z = \kappa_0 \mathbb{P}_{z^\perp} + (\kappa_0 + 2\kappa_1)\mathbb{P}_{\mathrm{i}z},
\end{eqnarray*}
where $\mathbb{P}_{\iota(z)}$ and $\mathbb{P}_{\iota(z)^\perp}$ are projection operators defined on $\bbr^{2d}$.
\end{lemma}
\begin{proof}
(i) The first two estimates follow directly from Definition \ref{D4.2}:
\begin{align*}
\begin{aligned}
\mathbb{P}_z^2(v) &= \mathbb{P}_z((z\cdot v)z)= (z\cdot (z\cdot v)z)z
=(z \cdot v)z = \mathbb{P}_z(v), \\
\mathbb{P}_{z^\perp}^2(v) &= \mathbb{P}_{z^\perp}(v - (z \cdot v)z) 
= v - (z \cdot v)z - (z \cdot (v - (z \cdot v)z)z \\
&= v - (z \cdot v)z - (z \cdot v - (z \cdot v))z
= v - (z \cdot v)z = \mathbb{P}_{z^\perp}(v).
\end{aligned}
\end{align*}

\noindent (ii) Recall that $\iota$ is an inclusion map defined by
\begin{eqnarray*}
\iota: \bbc^d\to \bbr^{2d},\quad
(z^1,\cdots, z^d)\mapsto (\mathrm{Re}~z^1, \mathrm{Im}~z^1, \cdots, \mathrm{Re}~z^d, \mathrm{Im}~z^d).
\end{eqnarray*}
From the definition of $\iota$, we get the following identity for any $v \in \mathbb{C}^d$,
\[
\iota^{-1}\circ \mathbb{P}_{\iota(z)}\circ \iota(v) = \iota^{-1}\bigg((\iota(z) \cdot \iota(v))\iota(z)\bigg) =(\iota(z) \cdot \iota(v))\iota^{-1}(\iota(z))
= (z \cdot v)z = \mathbb{P}_z (v).
\]
Similarly, one has 
\begin{align*}
\iota^{-1}\circ \mathbb{P}_{\iota(z)^\perp}\circ \iota(v) &= \iota^{-1}\bigg(\iota(v) - (\iota(z) \cdot \iota(v))\iota(z)\bigg)\\
&= \iota^{-1}(\iota(v)) - (\iota(z) \cdot \iota(v))\iota^{-1}(\iota(z))
=v - (z \cdot v)z = \mathbb{P}_{z^{\perp}}(v).
\end{align*}
Thus we have the desired result. \newline

\noindent (iii)~First, we substitute $J_\rho$ in \eqref{D-1-2} into $Q_z$ to see
\begin{align*}
Q_z(J_\rho ) & = \kappa_0(J_\rho - \langle J_\rho , z \rangle z) + \kappa_1(\langle z, J_\rho  \rangle - \langle J_\rho , z \rangle)z\\
&= \kappa_0\left(\int_{\mathbb{HS}^{d-1}}z_*\rho(z_*)d\sigma_{z_*} - \left\langle \int_{\mathbb{HS}^{d-1}}z_*\rho(z_*)d\sigma_{z_*}, z \right\rangle z\right)\\
&\hspace{0.2cm} + \kappa_1\left(\left\langle z, \int_{\mathbb{HS}^{d-1}}z_*\rho(z_*)d\sigma_{z_*} \right\rangle - \left\langle \int_{\mathbb{HS}^{d-1}}z_*\rho(z_*)d\sigma_{z_*}, z \right\rangle\right)z\\
&= \int_{\mathbb{HS}^{d-1}}[\kappa_0(z_*-\langle{z_*, z}\rangle z)+\kappa_1(\langle{z, z_*}\rangle z-\langle{z_*, z}\rangle z)]\rho(z_*)d\sigma_{z_*} \\
&= L[\rho](z).
\end{align*}

\noindent For the last estimate, we use Lemma \ref{L4.1} to get that for $v \in \mathbb{C}^d$, 
\begin{align*}
Q_z(v) &= \kappa_0(v - \langle v, z \rangle z) + \kappa_1(\langle z, v \rangle - \langle v, z \rangle)z \\
&= \kappa_0(v -(v \cdot z)z - \mathrm{i}(\mathrm{i}v \cdot z)z) + \kappa_1\mathrm{i}(\mathrm{i}z \cdot v - \mathrm{i}v \cdot z)z \\
&= \kappa_0(v -(v \cdot z)z) + (\kappa_0 + 2\kappa_1)(\mathrm{i}z \cdot v)\mathrm{i}z  \\
&=\kappa_0 \mathbb{P}_{z^\perp}(v) + (\kappa_0 + 2\kappa_1)\mathbb{P}_{\mathrm{i}z}(v).
\end{align*}
\end{proof}
\begin{lemma}\label{L4.3}
Let $\rho$ be a solution to system \eqref{D-1}. Then, we have
\[
\frac{d}{dt}\int_{\mathbb{HS}^{d-1}} \phi(z)\rho(z)dz = \int_{\mathbb{HS}^{d-1}} \nabla_z \phi(z) \cdot Q_z(J_\rho ) \rho(z)dz,
\]
for $\phi \in \mathcal{C}^1(\mathbb{HS}^{d-1})$.
\end{lemma}
\begin{proof}
By Lemma \ref{L4.2} (i), the continuity equation \eqref{D-1} becomes 
\[
\partial_t \rho+\nabla_z \cdot(\rho L[\rho]) = \partial_t \rho+\nabla_z \cdot(\rho Q_z(J_\rho )) = 0. 
\] 
We multiply $\phi$ and use integration by parts to find
\begin{align*}
 \int_{\mathbb{HS}^{d-1}}\big(\partial_t \rho +\nabla_z \cdot(\rho Q_z(J_\rho ))\big)\phi(z)dz & = 0
\end{align*}
or equivalently
\begin{align*}
\int_{\mathbb{HS}^{d-1}}\partial_t \rho(z) \phi(z)dz &= -\int_{\mathbb{HS}^{d-1}}\big(\nabla_z \cdot(\rho Q_z(J_\rho ))\big)\phi(z)dz. 
\end{align*}
By Proposition \ref{P3.1}, we can also simplify R.H.S. as follows:
\begin{align*}
 -\int_{\mathbb{HS}^{d-1}}\big(\nabla_z \cdot(\rho Q_z(J_\rho ))\big)\phi(z)dz&= \int_{\mathbb{HS}^{d-1}}\nabla_z(\phi(z))\cdot (\rho Q_z(J_\rho )) dz.
\end{align*}
Finally, we get the desired estimate:
\[
\frac{d}{dt}\int_{\mathbb{HS}^{d-1}} \phi(z)\rho(z)dz =\int_{\mathbb{HS}^{d-1}}\partial_t \rho(z) \phi(z)dz = \int_{\mathbb{HS}^{d-1}} \nabla_z \phi(z) \cdot Q_z(J_\rho ) \rho(z)dz.
\]
\end{proof}
\begin{remark}\label{R4.2}
Let $\rho$ be a solution to system \eqref{D-3}. Then, for any test function $\phi \in \mathcal{C}^1(\bbr\times \mathbb{HS}^{d-1})$, we have
\[
\int_{\mathbb{HS}^{d-1}} \phi(z)\partial_t\rho(z)dz = \int_{\mathbb{HS}^{d-1}} \nabla_z \phi(z) \cdot Q_z(J_\rho ) \rho(z)dz.
\]
\end{remark}
In next lemma, we provide a uniform bound for the second derivative of $R^2$. 
\begin{lemma}\label{L4.4}
Let $\rho$ be a solution to \eqref{D-3}. Then, we have
\begin{align*}
\begin{aligned}
& (i)~\frac{dR^2}{dt}=2\kappa_0 \int_{\mathbb{HS}^{d-1}}(\|J_\rho\|^2-|z\cdot J_\rho|^2)\rho(z)d\sigma_z\\
&\hspace{4cm} +2(\kappa_0 + 2\kappa_1)\int_{\mathbb{HS}^{d-1}}|(\mathrm{i}z)\cdot J_\rho|^2\rho(z)d\sigma_z. \\
& (ii)~\sup_{0 \leq t < \infty} \Big| \frac{d^2R^2}{dt^2} \Big| < \infty. 
\end{aligned}
\end{align*}
\end{lemma}
\begin{proof}
\noindent (i)~By \eqref{D-2-5}, one has
\[
R^2=J_\rho \cdot J_\rho,\quad\text{where}\quad J_\rho =\int_{\mathbb{HS}^{d-1}}z\rho(z)d\sigma_z.
\]
\noindent $\bullet$~(Estimate of $\frac{dJ_\rho}{dt}$):~for a fixed $e\in\bbc^d$, we set
\[
L_e(z)=e\cdot z.
\]
Then, it follows from Lemma \ref{L3.2} that 
\[
\nabla_z L_e(z)=e-(e\cdot z)z=\mathbb{P}_{z^\perp}(e).
\]
Then we have
\begin{align*}
e\cdot \frac{dJ_\rho}{dt} &=\frac{d}{dt} \int_{\mathbb{HS}^{d-1}}(e\cdot z)\rho(z)d\sigma_z
=\frac{d}{dt} \int_{\mathbb{HS}^{d-1}}L_e(z)\rho(z)d\sigma_z\\
&= \int_{\mathbb{HS}^{d-1}}(\nabla_z L_e(z))\cdot Q_z(J_\rho)\rho(z)d\sigma_z=\int_{\mathbb{HS}^{d-1}}\mathbb{P}_{z^\perp}(e)\cdot Q_z(J_\rho)\rho(z)d\sigma_z,
\end{align*}
where we used Lemma \ref{L4.2} in the third equality. From this result and the fact 
\[
Q_z(J_\rho)= \kappa_0 \mathbb{P}_{z^\perp}(J_\rho) + (\kappa_0 + 2\kappa_1)\mathbb{P}_{\mathrm{i}z}(J_\rho),
\]
we get
\begin{align}
\begin{aligned}\label{est order}
e\cdot \frac{d J_\rho}{dt}  &= \int_{\mathbb{HS}^{d-1}}\mathbb{P}_{z^\perp}(e) \cdot (\kappa_0 \mathbb{P}_{z^\perp}(J_\rho ) + (\kappa_0 + 2\kappa_1)\mathbb{P}_{\mathrm{i}z}(J_\rho ))\rho(z)d\sigma_z\\
&= \int_{\mathbb{HS}^{d-1}}\mathbb{P}_{z^\perp}(e) \cdot \kappa_0 \mathbb{P}_{z^\perp}(J_\rho )\rho(z)d\sigma_z + \int_{\mathbb{HS}^{d-1}}\mathbb{P}_{z^\perp}(e) \cdot (\kappa_0 + 2\kappa_1)\mathbb{P}_{\mathrm{i}z}(J_\rho )\rho(z)d\sigma_z\\
&= \kappa_0 \int_{\mathbb{HS}^{d-1}}\mathbb{P}_{z^\perp}(e) \cdot \mathbb{P}_{z^\perp}(J_\rho )\rho(z)d\sigma_z + (\kappa_0 + 2\kappa_1)\int_{\mathbb{HS}^{d-1}}\mathbb{P}_{z^\perp}(e) \cdot \mathbb{P}_{\mathrm{i}z}(J_\rho )\rho(z)d\sigma_z.
\end{aligned}
\end{align}
Since $e$ is arbitrary, we substitute $e=J_\rho$ to obtain
\begin{align*}
\begin{aligned}
J_\rho\cdot \frac{d J_\rho}{dt}&= \underbrace{\kappa_0 \int_{\mathbb{HS}^{d-1}}\mathbb{P}_{z^\perp}(J_\rho) \cdot \mathbb{P}_{z^\perp}(J_\rho )\rho(z)d\sigma_z}_{=:\mathcal{J}_1} + \underbrace{(\kappa_0 + 2\kappa_1)\int_{\mathbb{HS}^{d-1}}\mathbb{P}_{z^\perp}(J_\rho) \cdot \mathbb{P}_{\mathrm{i}z}(J_\rho )\rho(z)d\sigma_z}_{=:\mathcal{J}_2}.
\end{aligned}
\end{align*}
Next, we estimate $\mathcal{J}_1$ and $\mathcal{J}_2$ separately.\\

\noindent$\diamond$ (Estimate of $\mathcal{J}_1$): By direct calculations, we have
\begin{align*}
\mathcal{J}_1&=\kappa_0 \int_{\mathbb{HS}^{d-1}}\mathbb{P}_{z^\perp}(J_\rho) \cdot \mathbb{P}_{z^\perp}(J_\rho )\rho(z)d\sigma_z\\
&=\kappa_0 \int_{\mathbb{HS}^{d-1}}\|J_\rho-(z\cdot J_\rho)z\|^2\rho(z)d\sigma_z
=\kappa_0 \int_{\mathbb{HS}^{d-1}}(\|J_\rho\|^2-|z\cdot J_\rho|^2)\rho(z)d\sigma_z,
\end{align*}
where we use the fact that $z\cdot J_\rho$ is a real number.

\vspace{0.5cm}

\noindent$\diamond$ (Estimate of $\mathcal{J}_2$): First, we simplify the part of an integrand in $\mathcal{J}_2$ as follows. Since $z\cdot w$ is real number for all $z, w\in\bbc^d$, we have
\begin{align*}
\mathbb{P}_{z^\perp}(J_\rho) \cdot \mathbb{P}_{\mathrm{i}z}(J_\rho )&=(J_\rho-(z\cdot J_\rho)z) \cdot(((\mathrm{i}z)\cdot J_\rho)(\mathrm{i}z))\\
&=J_\rho\cdot (((\mathrm{i}z)\cdot J_\rho)\mathrm{i}z)-((z\cdot J_\rho)z)\cdot(((\mathrm{i}z)\cdot J_\rho)(\mathrm{i}z))\\
&=((\mathrm{i}z)\cdot J_\rho)((\mathrm{i}z)\cdot J_\rho)-(z\cdot J_\rho)((\mathrm{i}z)\cdot J_\rho)(z\cdot(\mathrm{i}z))\\
&=|(\mathrm{i}z)\cdot J_\rho|^2,
\end{align*}
where we used $z\cdot(\mathrm{i}z)=0$ in the last equality. Thus we have
\begin{align*}
\mathcal{J}_2&=(\kappa_0 + 2\kappa_1)\int_{\mathbb{HS}^{d-1}}|(\mathrm{i}z)\cdot J_\rho|^2\rho(z)d\sigma_z.
\end{align*}
We combine all the estimates for $\mathcal{J}_1$ and $\mathcal{J}_2$ to get
\[
J_\rho\cdot \frac{dJ_\rho}{dt} =\kappa_0 \int_{\mathbb{HS}^{d-1}}( \|J_\rho\|^2-|z\cdot J_\rho|^2)\rho(z)d\sigma_z+(\kappa_0 + 2\kappa_1)\int_{\mathbb{HS}^{d-1}}|(\mathrm{i}z)\cdot J_\rho|^2\rho(z)d\sigma_z.
\]
This yields 
\[
\frac{dR^2}{dt}=2\kappa_0 \int_{\mathbb{HS}^{d-1}}(\|J_\rho\|^2-|z\cdot J_\rho|^2)\rho(z)d\sigma_z+2(\kappa_0 + 2\kappa_1)\int_{\mathbb{HS}^{d-1}}|(\mathrm{i}z)\cdot J_\rho|^2\rho(z)d\sigma_z \geq 0,
\]
where we used $|z\cdot J_\rho|^2 \leq |z|^2 \cdot \|J_\rho \|^2 = \|J_\rho \|^2. $

\vspace{0.5cm}

\noindent (ii)~We leave its proof in Appendix \ref{App-D}.
\end{proof}
\begin{remark}
By Lemma 4.1 of \cite{H-P2}, if $\{z_j\}$ is a global solution to the following system:
\[
\dot{z}_j=\frac{\kappa_0}{N}\sum_{k=1}^N(z_k-\langle z_k, z_j\rangle z_j)+\frac{\kappa_1}{N}\sum_{k=1}^N(\langle z_j, z_k\rangle-\langle z_k, z_j\rangle)z_j,
\]
then we have
\[
\frac{d}{dt}\|z_c\|^2=\frac{2\kappa_0}{N}\sum_{i=1}^N(\|z_c\|^2-|\langle z_i, z_c\rangle|^2)+\frac{4(\kappa_0+\kappa_1)}{N}\sum_{i=1}^N|\mathrm{Im}\langle z_i, z_c\rangle|^2,
\]
where $z_c=\frac{1}{N}\sum_{j=1}^Nz_j$. This and the relation
\[
|\langle z_i, z_c\rangle|^2=|\mathrm{Re}\langle z_i, z_c\rangle|^2+|\mathrm{Im}\langle z_i, z_c\rangle|^2,
\]
imply the desired result:
\[
\frac{d}{dt}\|z_c\|^2=\frac{2\kappa_0}{N}\sum_{i=1}^N(\|z_c\|^2-|\mathrm{Re}\langle z_i, z_c\rangle|^2)+\frac{2(\kappa_0+2\kappa_1)}{N}\sum_{i=1}^N|\mathrm{Im}\langle z_i, z_c\rangle|^2.
\]
\end{remark}
\subsection{Proof of Theorem \ref{T4.1}} \label{sec:4.3}
It follows from Lemma \ref{L4.3} and Lemma \ref{L4.4} that $R^2 = \|J_\rho\|^2$ is increasing and its second derivative is uniformly bounded. Then we can apply Barbalat's lemma \cite{Ba} to obtain
\[
\lim_{t\to\infty}\frac{d}{dt}\|J_{\rho}\|^2=0.
\]
Note that $\frac{d}{dt}\|J_\rho\|^2$ can be expressed as follows:
\[
\frac{d\|J_\rho\|^2}{dt} =2\kappa_0 \int_{\mathbb{HS}^{d-1}}(\|J_\rho\|^2-|z\cdot J_\rho|^2)\rho(z)d\sigma_z+2(\kappa_0 + 2\kappa_1)\int_{\mathbb{HS}^{d-1}}|(\mathrm{i}z)\cdot J_\rho|^2\rho(z)d\sigma_z.
\]
Since
\[
2\kappa_0 \int_{\mathbb{HS}^{d-1}}(\|J_\rho\|^2-|z\cdot J_\rho|^2)\rho(t,z)d\sigma_z \geq0, \quad 
2(\kappa_0 + 2\kappa_1)\int_{\mathbb{HS}^{d-1}}|(\mathrm{i}z)\cdot J_\rho|^2\rho(t,z)d\sigma_z \geq0,
\]
we have the desired estimate:
\[
\lim_{t\to\infty} \int_{\mathbb{HS}^{d-1}}( \|J_\rho\|^2-|z\cdot J_\rho|^2)\rho(z)d\sigma_z=0.
\]

\vspace{0.5cm}

As a corollary of Theorem \ref{T4.1}, one has complete aggregation.

\begin{corollary}[Emergence of complete aggregation]  Suppose system parameters and initial measure satisfy 
\[
|\kappa_1| < \frac{\kappa_0}{2},\quad 0\leq\sup_{z, w \in \mathrm{supp}(\mu_0)}|1-\langle z, w\rangle| < 1 - \frac{2|\kappa_1|}{\kappa_0}-\delta,\quad  \Omega_j\equiv \Omega,
\]
for some positive constant $\delta$, and let $\rho$ be a solution to \eqref{D-3}. Then there exists a curve $z(t)$ on $\mathbb{HS}^{d-1}$ such that for the probability measure $\rho(t) d\sigma_z$ and the dirac measure $\delta_{z(t)}$, we have
\[
\lim_{t\to\infty}W_2(\rho(t) d\sigma_z, \delta_{z(t)}) = 0.
\]
\end{corollary}
\begin{proof}
For notational simplicity, we take $\mu_t$ as $\rho(t) d\sigma_z$. Then, by the same procedure as in the proof of the second part of Proposition \ref{T3.1}, there exists a sequence of empirical probability measures $\{\mu^{N}\}_{N\in\mathbb{N}}$ such that 
\begin{align}
\begin{aligned} \label{p1}
& \lim_{N\to\infty}W_2(\mu_t, \mu_t^{N}) = 0, \quad t>0, \quad and \quad \mathrm{supp}(\mu^N_0) \subset \mathrm{supp}(\mu_0),\\
& \mu \text{ is derivable from } \mu^N \text{ in } [0,\infty) \text{ with respect to } W_2-\text{metric.} 
\end{aligned}
\end{align}
By \eqref{p1}, for any $\varepsilon > 0$, there exists a positive integer $N$ such that 
\[
W_2(\mu_t, \mu_t^{n}) < \varepsilon \quad \text{ for } \quad n \geq N, \text{ independent of } t \in [0,\infty).
\]
Again, by the a priori condition, for each $N$, the empirical measure corresponds to $N$-particles and the dynamics of $N$-particles following the LHS model exhibits the complete aggregation: there exists path $z^N:[0, \infty)\to\mathbb{HS}^{d-1}$ which satisfies 
\[
\lim_{t\to\infty}\|z_i(t)-z^N(t)\|=0\quad\text{for all}~i\in\mathcal{N}, \quad \text{or equivalently} \quad \lim_{t\to\infty}W_2(\mu^N(t), \delta_{z^N(t)})=0,
\]
where $Z=\{z_i\}$ be a solution of the LHS model. We use the triangle inequality to find
\[
W_2(\mu_t, \delta_{z^N(t)}) \leq W_2(\mu_t, \mu_t^N) + W_2(\mu_t^N, \delta_{z^N(t)}) < \varepsilon + W_2(\mu_t^N, \delta_{z^N(t)}).
\]
Hence, we get
\begin{align}\label{D-10}
\limsup_{t\to\infty}W_2(\mu_t, \delta_{z^N(t)}) \leq \varepsilon.
\end{align}
We prove the statement via the proof by contradiction. Suppose that $\mu_t$ weakly converges to the measure $m \delta_{z(t)} + (1-m) \delta_{-z(t)}$ for $0<m<1$. Then we have
\begin{align*}
W_2(\mu_t, m \delta_{z(t)} + (1-m) \delta_{-z(t)}) &\geq W_2(\delta_{z_N(t)}, m \delta_{z(t)} + (1-m) \delta_{-z(t)}) - W_2(\delta_{z_N(t)}, \mu_t).
\end{align*}
Since the only measure which takes $\delta_x$ and $m\delta_y + (1-m)\delta_z$ as marginals is 
\[
m\delta_x \otimes \delta_y + (1-m)\delta_x \otimes \delta_z,
\]
we have
\begin{align*}
\begin{aligned}
& W^2_2(\delta_{z^N(t)}, m \delta_{z(t)} + (1-m) \delta_{-z(t)}) \\
& \hspace{0.5cm} = m\| z^N(t) - z(t) \|^2 + (1-m)\| z^N(t) - \big(-z(t)\big) \|^2 \\
& \hspace{0.5cm}  = \|z^N(t)\|^2+\|z(t)\|^2+2\mathrm{Re}((1-m)\langle z^N(t), z(t) \rangle - m\langle z^N(t), z(t) \rangle) \\
& \hspace{0.5cm} = 2 + 2(1-2m)\mathrm{Re}(\langle z^N(t), z(t) \rangle) \geq 2-2|1-2m| = 4\min\{m, 1-m\}.
\end{aligned}
\end{align*}
Therefore, for any $0<m<1$, we first choose $\varepsilon$ to satisfy
\[
0<\varepsilon<2 \sqrt{\min\{m, 1-m\}}.
\]
Then, there exists sufficiently large $N$ satisfying \eqref{D-10}, so that
\[
2\sqrt{\min\{m, 1-m\}}\leq \limsup_{t\to\infty}W_2(\mu(t), \delta_{z^N(t)})\leq \varepsilon,
\]
which gives a contradiction. So we can conclude either $m=0$ or $m=1$. Therefore, a bi-polar state cannot emerge.
\end{proof}

\section{Conclusion} \label{sec:5}
\setcounter{equation}{0} 
In this paper, we have studied the well-posedness and emergent behaviors of the kinetic mean-field model for the LHS model which corresponds to the complex analogue of the LS model for aggregation. The LS model is the first-order aggregation model on the unit sphere in Euclidean space. The LHS model also reduces to the LS model on the unit sphere in complex Euclidean space. The kinetic LHS model can be formally obtained via the BBGKY hierarchy from the LHS model. In this work, we have discussed three issues. First, we presented exponential aggregation estimates to the LHS model with relaxed coupling strengths for some restricted class of initial data. Second, we provided a global-in-time well-posedness of measure-valued solutions to the kinetic LHS model using the particle-in-cell method and uniform-stability estimate. Third, we presented the emergent dynamics of the kinetic LHS model. In general, the particle-in-cell method provides a measure-valued solution in any finite-time interval for generic initial data. However, this type of mean-field limit cannot be extended to the whole time interval due to the lack of a suitable uniform stability estimate. As long as the initial data and system parameters satisfy some admissible conditions, we can see that the uniform stability estimate follows, and the kinetic LHS model can be derivable uniformly in time from the LHS model. Of course, there are several issues to be discussed in future work. For example, in our work, the uniform stability estimate was obtained for some admissible class of initial data and system parameters. Thus, the extension of uniform stability to a relaxed setting will remain an interesting problem. In addition, our emergent dynamics in Section \ref{sec:4} has been studied only for the homogeneous ensemble with the same natural frequency matrices. Clearly, extension to a heterogeneous ensemble will be interesting to pursue in future work.

\appendix

\section{Proof of Theorem \ref{T2.3}}\label{App-A}
\setcounter{equation}{0}
It follows from the solution splitting property of the LHS model that we can set $\Omega\equiv 0$ without loss of generality, and $z_j$ satisfies 
\begin{equation} \label{XX-0-1}
\begin{cases}
\dot{z}_j= \kappa_0 \Big (\langle{z_j, z_j}\rangle z_c-\langle{z_c, z_j}\rangle z_j \Big )+\kappa_1 \Big (\langle{z_j, z_c}\rangle-\langle{z_c, z_j}\rangle \Big )z_j, \quad t > 0, \\
z_j(0)=z_j^0, \quad  \|z^0_j \| = 1, \quad j \in {\mathcal N}.
\end{cases}
\end{equation}
Next, we introduce two-point correlation like functionals:
\begin{equation} \label{XX-0-2}
h_{ij} := \langle z_i, z_j \rangle, \quad R_{ij} = \mathrm{Re}(h_{ij}), \quad  I_{ij} = \mathrm{Im}(h_{ij}), \quad i, j \in {\mathcal N} \cup \{c \}.
\end{equation}
Here, the index $c$ stands for the centroid of the ensemble. Then, we can rewrite system $\eqref{XX-0-1}_1$ using \eqref{XX-0-2} as 
\begin{align}
\begin{aligned}\label{XX-1}
\dot{z}_i 
&= \frac{\kappa_0}{N}\sum_{k=1}^N (z_k-\langle z_k, z_i \rangle z_i)+ \frac{\kappa_1}{N}\sum_{k=1}^N (\langle z_i, z_k \rangle - \langle z_k, z_i \rangle)z_i
\\ &= \kappa_0(z_c-h_{ci} z_i)+ \kappa_1(h_{ic} - h_{ci})z_i = \kappa_0\big(z_c - (R_{ci} + \mathrm{i}I_{ci})z_i\big) + 2\mathrm{i}\kappa_1 I_{ic} z_i 
\\ &= \kappa_0 z_c - \big( \kappa_0 R_{ci} + \mathrm{i} (\kappa_0 + 2\kappa_1 )I_{ci} \big)z_i.
\end{aligned}
\end{align}
Here, we used the relation $I_{ci}=-I_{ic}$ in the last equality. Again, we use \eqref{XX-1} to derive the time derivative of $h_{ij}$:
\begin{align} 
\begin{aligned} \label{XX-2}
\dot{h}_{ij} &= \langle \dot{z}_i, z_j \rangle + \langle z_i, \dot{z}_j \rangle
\\ &= \left\langle \kappa_0 z_c - \big( \kappa_0 R_{ci} + \mathrm{i} (\kappa_0 + 2\kappa_1 )I_{ci} \big)z_i, z_j \right\rangle + \left\langle z_i, \kappa_0 z_c - \big( \kappa_0 R_{cj} + \mathrm{i} (\kappa_0 + 2\kappa_1) I_{cj} \big)z_j \right\rangle
\\ &= \kappa_0 h_{cj} - \big(\kappa_0 R_{ci} - \mathrm{i} (\kappa_0 + 2\kappa_1) I_{ci}\big)h_{ij}
+ \kappa_0 h_{ic} - \big( \kappa_0 R_{cj} + \mathrm{i} (\kappa_0 + 2\kappa_1) I_{cj} \big) h_{ij}
\\ &= \kappa_0 (h_{cj} + h_{ic} - R_{ci}h_{ij} - R_{cj}h_{ij})
+ \mathrm{i}(\kappa_0+ 2\kappa_1)(I_{ci}-I_{cj})h_{ij}.
\end{aligned} 
\end{align}
Now, we take the real and imaginary parts of  \eqref{XX-2} to find
\begin{align}
\begin{aligned}\label{XX-3}
\dot{R}_{ij} &= \kappa_0 (R_{cj} + R_{ic} - R_{ci}R_{ij} - R_{cj}R_{ij})
- (\kappa_0+ 2\kappa_1)(I_{ci}-I_{cj})I_{ij} 
\\ &= \kappa_0 \big((R_{jc} + R_{ic})(1 - R_{ij})\big)
- (\kappa_0+ 2\kappa_1)(I_{ci}-I_{cj})I_{ij}, \\
\dot{I}_{ij} &= \kappa_0 (I_{cj} + I_{ic} - R_{ci}I_{ij} - R_{cj}I_{ij})
+ (\kappa_0+ 2\kappa_1)(I_{ci}-I_{cj})R_{ij} 
\\ &= (I_{cj} + I_{ic})\big(\kappa_0 - (\kappa_0 + 2\kappa_1) R_{ij} \big)
- \kappa_0 I_{ij}(R_{ci} + R_{cj}).
\end{aligned}
\end{align}
Next, we note that if complete aggregation occurs, then one has 
\[ z_i - z_j \to 0 \quad \mbox{as $t \to \infty$} \quad \mbox{hence} \quad h_{ij} \to 1 \quad \mbox{as $t \to \infty$}, \] 
or equivalently, 
\[   (R_{ij}, I_{ij}) \to (1,0), \quad \mbox{as $t \to \infty$}. \]
Thus, it is convenient to work with the following functional $J_{ij}$ :
\[
J_{ij} := 1 - R_{ij}.
\]
We substitute $R_{ij}=1-J_{ij}$ into relations \eqref{XX-3}  to get
\begin{align*}
\dot{J}_{ij} &= -\dot{R}_{ij}
= -\kappa_0 (2-J_{jc} - J_{ic})J_{ij} + (\kappa_0+ 2\kappa_1)(I_{ci}-I_{cj})I_{ij},\\
\dot{I}_{ij} &= (I_{cj} + I_{ic})\big(-2\kappa_1 + (\kappa_0 + 2\kappa_1)J_{ij} \big)
- \kappa_0 I_{ij}(2 - J_{ci} - J_{cj}).
\end{align*} 
Therefore, we have
\begin{align}
\begin{aligned}\label{XX-5}
&\frac{1}{2}\frac{d}{dt}(I_{ij}^2+J_{ij}^2) = \dot{I}_{ij}I_{ij}+\dot{J}_{ij}J_{ij}\\ 
&\hspace{0.5cm} = \Big((I_{cj} + I_{ic})\big(-2\kappa_1 + (\kappa_0 + 2\kappa_1)J_{ij} \big)
- \kappa_0 I_{ij}(2 - J_{ci} - J_{cj})\Big)I_{ij} \\ 
&\hspace{0.5cm}- \Big(\kappa_0 (2-J_{jc} - J_{ic})J_{ij} - (\kappa_0+ 2\kappa_1)(I_{ci}-I_{cj})I_{ij}\Big)J_{ij} \\ 
&\hspace{0.5cm}= -2\kappa_1\big((I_{cj} + I_{ic}) - \kappa_0 I_{ij}(2 - J_{ci} - J_{cj})\big)I_{ij}
- \big(\kappa_0 (2-J_{jc} - J_{ic})J_{ij}\big)J_{ij} \\
&\hspace{0.5cm}= -\kappa_0(2-J_{jc}-J_{ic})(I_{ij}^2 +J_{ij}^2) - 2\kappa_1(I_{cj}+I_{ic})I_{ij}.
\end{aligned}
\end{align}
For $t \geq 0$, we choose time-dependent indices $i=i_t$ and $j=j_t$ such that 
\begin{align}\label{XX-6}
(i,j) \in \argmax_{(k,l)} (I_{kl}^2 + J_{kl}^2).
\end{align}
Then, one has
\begin{align}\label{XX-7}
| I_{\alpha c} | \leq \max_{\beta,\gamma} | I_{\beta \gamma} | \leq \sqrt{I_{ij}^2 + J_{ij}^2}, \quad
 J_{\alpha c}  \leq \max_{\beta ,\gamma}  J_{\beta \gamma} \leq \sqrt{I_{ij}^2 + J_{ij}^2},
\end{align}
for any index $\alpha$.  It follows from \eqref{XX-5} and \eqref{XX-7} that 
\begin{align}\begin{aligned}\label{XX-8}
\frac{1}{2}\frac{d}{dt}(I_{ij}^2+J_{ij}^2)&= -\kappa_0(2-J_{jc}-J_{ic})(I_{ij}^2 +J_{ij}^2) - 2\kappa_1(I_{cj}+I_{ic})I_{ij} &&\big(\mbox{by}~ \eqref{XX-5}\big)
\\ & \leq -\kappa_0(2-J_{jc}-J_{ic})(I_{ij}^2 +J_{ij}^2) + 2|\kappa_1|(| I_{cj} |+| I_{ic} |)| I_{ij} |
\\ &\leq -\kappa_0(2-J_{jc}-J_{ic})(I_{ij}^2 +J_{ij}^2) + 4|\kappa_1|(I_{ij}^2 + J_{ij}^2) &&\big(\mbox{by}~ \eqref{XX-7}_1\big)
\\ & =-\kappa_0\left(2-J_{jc}-J_{ic}-\frac{4|\kappa_1|}{\kappa_0}\right)(I_{ij}^2+J_{ij}^2)
\\ & =-2\kappa_0\left(1-\sqrt{I_{ij}^2+J_{ij}^2}-\frac{2|\kappa_1|}{\kappa_0}\right)(I_{ij}^2+J_{ij}^2) &&\big( \mbox{by}~ \eqref{XX-7}_2 \big).
\end{aligned}
\end{align}
On the other hand, by \eqref{XX-6} one has 
\[
\mathcal{F} = \max_{k, l}|1-\langle z_k, z_l\rangle|=\sqrt{\max_{k,l}(I_{kl}^2+J_{kl}^2)} \overset{\eqref{XX-6}}{=} \sqrt{I_{ij}^2+J_{ij}^2}.
\]
Then we can rewrite \eqref{XX-8} and obtain the following inequality:
\begin{align*}\label{XX-9}
\dot{\mathcal{F}} \leq -2\kappa_0\left(1-\mathcal{F}-\frac{2|\kappa_1|}{\kappa_0}\right)\mathcal{F}.
\end{align*}
If $0\leq\mathcal{F}^0<1-\frac{2|\kappa_1|}{\kappa_0}-\delta$ for some positive $\delta>0$, then we can see that $\mathcal{F}$ is decreasing for all $t\geq0$. So, we obtain Gronwall's inequality:
\[
\dot{\mathcal{F}}\leq -2\kappa_0\left(1-\mathcal{F}^0-\frac{2|\kappa_1|}{\kappa_0}\right)\mathcal{F}<-2\kappa_0\delta \mathcal{F}.
\]
This implies the exponential decay of $\mathcal{F}$ :
\[
\mathcal{F}(t)< \mathcal{F}^0\exp(-2\kappa_0\delta t), \quad t \geq 0.
\]
\qed
\section{Proof of Proposition \ref{Lp_stability}}\label{App-B}
\setcounter{equation}{0}
\noindent (i)~Let $z_i$ and ${\tilde z}_i$ be solutions to \eqref{A-0}.  Then, $z_i$ and ${\tilde z}$ satisfy 
\begin{align*}
\dot{z}_i&=\Omega_i z_i + \frac{\kappa_0}{N}\sum_{k=1}^N (z_k-\langle z_k, z_i \rangle z_i)+ \frac{\kappa_1}{N}\sum_{k=1}^N (\langle z_i, z_k \rangle - \langle z_k, z_i \rangle)z_i\\
&= \Omega_i z_i + \frac{\kappa_0}{N}\sum_{k=1}^N (z_k-R_{ik} z_i)+ 2 \mathrm{i}\frac{(\frac{\kappa_0}{2}+\kappa_1)}{N}\sum_{k=1}^N I_{ik}z_i,\\
\end{align*}
and
\begin{align*}
\dot{{\tilde z}}_i &=  \Omega_i {\tilde z}_i + \frac{\kappa_0}{N}\sum_{k=1}^N ({\tilde z}_k- {\tilde R}_{ik} {\tilde z}_i)+ 2 \mathrm{i}\frac{(\frac{\kappa_0}{2}+\kappa_1)}{N}\sum_{k=1}^N {\tilde I}_{ik} {\tilde z}_i,
\end{align*} 
where $R_{ij}=\mathrm{Re}\langle z_i, z_j \rangle$, $I_{ij}=\mathrm{Im}\langle z_i, z_j \rangle$. Then, $\tilde{z}_i-z_i$ satisfies
\begin{align*}
&\frac{d}{dt}(\tilde{z}_i-z_i) \\
&\hspace{0.5cm} =\Omega_i(\tilde{z}_i-z_i) + \frac{\kappa_0}{N}\sum_{k=1}^N \big((\tilde{z}_k-z_k)-(\tilde{R}_{ik}\tilde{z}_i-R_{ik} z_i)\big) + 2 \mathrm{i}\frac{(\frac{\kappa_0}{2}+\kappa_1)}{N}\sum_{k=1}^N (\tilde{I}_{ik}\tilde{z}_i-I_{ik}z_i)\\
&\hspace{0.5cm}= \Omega_i(\tilde{z}_i-z_i)
+ \frac{\kappa_0}{N}\sum_{k=1}^N (\tilde{z}_k-z_k)
 - \frac{\kappa_0}{N}\sum_{k=1}^N \big(\tilde{R}_{ik}(\tilde{z}_i-z_i)+(\tilde{R}_{ik}-R_{ik}) z_i\big) \\
&\hspace{1cm}+ 2 \mathrm{i} \frac{(\frac{\kappa_0}{2}+\kappa_1)}{N}\sum_{k=1}^N (\tilde{I}_{ik}\tilde{z}_i-I_{ik}z_i).
\end{align*}
This yields
\begin{align*}
&\left\langle \frac{d}{dt}(\tilde{z}_i-z_i), \tilde{z}_i-z_i \right\rangle \\
& \hspace{0.5cm} =\left\langle \Omega_i(\tilde{z}_i-z_i), \tilde{z}_i-z_i \right\rangle
+ \frac{\kappa_0}{N}\sum_{k=1}^N \langle \tilde{z}_k-z_k, \tilde{z}_i-z_i \rangle - \frac{\kappa_0}{N}\sum_{k=1}^N \left\langle \tilde{R}_{ik}(\tilde{z}_i-z_i)+(\tilde{R}_{ik}-R_{ik}) z_i, \tilde{z}_i-z_i \right\rangle \\
& \hspace{0.5cm}  - 2 \mathrm{i} \frac{(\frac{\kappa_0}{2}+\kappa_1)}{N}\sum_{k=1}^N \langle \tilde{I}_{ik}\tilde{z}_i-I_{ik}z_i, \tilde{z}_i - z_i \rangle\\
& \hspace{0.5cm} = \left\langle \Omega_i(\tilde{z}_i-z_i), \tilde{z}_i-z_i \right\rangle
+ \frac{\kappa_0}{N}\sum_{k=1}^N \left\langle \tilde{z}_k-z_k, \tilde{z}_i-z_i \right\rangle - \frac{\kappa_0}{N}\sum_{k=1}^N ( \tilde{R}_{ik}\|\tilde{z}_i-z_i\|^2 + (\tilde{R}_{ik}-R_{ik})\langle z_i, \tilde{z}_i-z_i \rangle) \\
& \hspace{0.5cm}- 2 \mathrm{i} \frac{(\frac{\kappa_0}{2}+\kappa_1)}{N}\sum_{k=1}^N \big(\tilde{I}_{ik}\langle\tilde{z}_i, \tilde{z}_i - z_i \rangle-I_{ik}\langle z_i, \tilde{z}_i - z_i \rangle\big).
\end{align*}
Since $\Omega_i$ is skew-Hermitian, we can see that $\langle \Omega_i(\tilde{z}_i-z_i), \tilde{z}_i-z_i \rangle$ is purely imaginary. Thus, one has
\begin{align}\label{Y-0}
\begin{aligned}
\frac{d}{dt}\|\tilde{z}_i-z_i\|^2&= \left\langle \frac{d}{dt}(\tilde{z}_i-z_i), \tilde{z}_i-z_i \right\rangle + \left\langle \tilde{z}_i-z_i, \frac{d}{dt}(\tilde{z}_i-z_i) \right\rangle\\
& = 2\mathrm{Re}\left\langle \frac{d}{dt}(\tilde{z}_i-z_i), \tilde{z}_i-z_i \right\rangle \\
&= \frac{2\kappa_0}{N}\sum_{k=1}^N \left( \mathrm{Re}\langle \tilde{z}_k-z_k, \tilde{z}_i-z_i \rangle - \tilde{R}_{ik} \| \tilde{z}_i - z_i \|^2 + (R_{ik}-\tilde{R}_{ik})\mathrm{Re}\langle z_i, \tilde{z}_i-z_i \rangle \right)\\
&\hspace{0.2cm} + \frac{4(\frac{\kappa_0}{2}+\kappa_1)}{N}\sum_{k=1}^N
\underbrace{\mathrm{Im}\left(\tilde{I}_{ik}\langle \tilde{z}_i , \tilde{z}_i-z_i \rangle - I_{ik} \langle z_i, \tilde{z}_i-z_i \rangle \right)}_{=: \mathcal{I}_{i,k}}.
\end{aligned}
\end{align}
 Note that 
\begin{align}\label{Y-1}
\begin{aligned}
\mathcal{I}_{i,k} &= 
\mathrm{Im}\left(\tilde{I}_{ik}(\underbrace{\langle \tilde{z}_i , \tilde{z}_i \rangle}_{=1} - \langle \tilde{z}_i , z_i \rangle) - I_{ik}(\langle z_i, \tilde{z}_i \rangle-\underbrace{\langle z_i, z_i \rangle}_{=1}) \right)
= \mathrm{Im}\left(-\tilde{I}_{ik}\langle \tilde{z}_i , z_i \rangle - I_{ik}\langle z_i, \tilde{z}_i \rangle \right)\\
&= - \mathrm{Im}\left(\underbrace{\tilde{I}_{ik}(\langle \tilde{z}_i , z_i \rangle + \langle z_i, \tilde{z}_i \rangle)}_{=\text{real number}} + (I_{ik}-\tilde{I}_{ik})\langle z_i, \tilde{z}_i \rangle \right)
= -(I_{ik}-\tilde{I}_{ik})\mathrm{Im}\langle z_i, \tilde{z}_i \rangle.
\end{aligned}
\end{align}
We combine \eqref{Y-0} and \eqref{Y-1} to obtain
\begin{align*}
\frac{d}{dt}\|\tilde{z}_i-z_i\|^2
&= \frac{2\kappa_0}{N}\sum_{k=1}^N \left( \mathrm{Re}\langle \tilde{z}_k-z_k, \tilde{z}_i-z_i \rangle - \tilde{R}_{ik} \| \tilde{z}_i - z_i \|^2 + (R_{ik}-\tilde{R}_{ik})\mathrm{Re}(\langle z_i, \tilde{z}_i \rangle - 1 ) \right)\\
&\hspace{0.2cm}- \frac{4(\frac{\kappa_0}{2}+\kappa_1)}{N}\sum_{k=1}^N (I_{ik}-\tilde{I}_{ik})\mathrm{Im}\langle z_i, \tilde{z}_i \rangle.
\end{align*}
Now, we use 
\[
\big|\mathrm{Re}\langle \tilde{z}_k-z_k, \tilde{z}_i-z_i \rangle\big| \leq \| \tilde{z}_k - z_k \| \cdot\| \tilde{z}_i - z_i \|,\quad \mathrm{Re}\langle z_i, \tilde{z}_i \rangle - 1 = -\frac{\|z_i - \tilde{z}_i\|^2}{2}
\]
to get 
\begin{align*}
\begin{aligned}
\frac{d}{dt}\|\tilde{z}_i-z_i\|^2 &\leq \frac{2|\kappa_0|}{N}\sum_{k=1}^N \left( \|\tilde{z}_k-z_k\| \cdot\|\tilde{z}_i-z_i\| +\|z_i - \tilde{z}_i\|^2  \right) \\
&+ \frac{2|\kappa_0+2\kappa_1|}{N}\sum_{k=1}^N \big|(I_{ik}-\tilde{I}_{ik})\cdot \mathrm{Im}\langle z_i, \tilde{z}_i \rangle\big|.
\end{aligned}
\end{align*}
To bound the last term, we use the Cauchy-Schwarz inequality and unit modulus of $z_i$ and $\tilde{z}_i$ to get
\begin{align*}
\begin{aligned}
| I_{ik}-\tilde{I}_{ik} |  &= | \mathrm{Im}(\langle z_i, z_k \rangle-\langle \tilde{z}_i, \tilde{z}_k \rangle) | \\
&= | \mathrm{Im}(\langle z_i-\tilde{z}_i, z_k \rangle + \langle \tilde{z}_i, z_k - \tilde{z}_k \rangle) | \leq \|z_i - \tilde{z}_i\|+\|z_k-\tilde{z}_k\|, \\
| \mathrm{Im}\langle z_i, \tilde{z}_i \rangle | &= |\mathrm{Im}\langle z_i-\tilde{z}_i, \tilde{z}_i\rangle|\leq |\langle z_i-\tilde{z}_i, \tilde{z}_i\rangle| \leq \|z_i - \tilde{z}_i \|.
\end{aligned}
\end{align*}
To sum up, we have
\begin{align*}
\frac{d}{dt}\|\tilde{z}_i-z_i\| &\leq \frac{|\kappa_0|}{N}\sum_{k=1}^N \left( \|\tilde{z}_k-z_k\| +\|z_i - \tilde{z}_i\|  \right) + \frac{|\kappa_0+2\kappa_1|}{N}\sum_{k=1}^N (\|z_i - \tilde{z}_i\|+\|z_k-\tilde{z}_k\|)\\
&=\frac{|\kappa_0|+|\kappa_0+2\kappa_1|}{N}\sum_{k=1}^N (\|z_i - \tilde{z}_i\|+\|z_k-\tilde{z}_k\|).
\end{align*}
We multiply $p\|\tilde{z}_i-z_i\|^{p-1}$ to the above relation and sum up the resulting relation over all  $i \in {\mathcal N}$  to obtain
\begin{align*}
\frac{d}{dt}\sum_{i=1}^N\|\tilde{z}_i-z_i\|^p
\leq \frac{p(|\kappa_0|+|\kappa_0+2\kappa_1|)}{N}\sum_{i,k=1}^N \bigg( \|\tilde{z}_k-z_k\|\cdot\|\tilde{z}_i-z_i\|^{p-1} +\|z_i - \tilde{z}_i\|^p  \bigg).
\end{align*}
By H\"{o}lder type inequalities, one has 
\begin{align*}
&\sum_{k=1}^N\|\tilde{z}_k-z_k\| =\|\tilde{Z}-Z\|_1 \leq N^{\frac{p-1}{p}}\|\tilde{Z}-Z\|_p,\\
&\sum_{i=1}^N\|\tilde{z}_i-z_i\|^{p-1} =\|\tilde{Z}-Z\|_{p-1}^{p-1} \leq \Big(\|\tilde{Z}-Z\|_p \cdot N^{\frac{1}{p(p-1)}} \Big)^{p-1}=N^{\frac{1}{p}}\|\tilde{Z}-Z\|_p^{p-1}.
\end{align*}
These imply
\begin{align}\label{Y-2}
\sum_{i, k=1}^N\|\tilde{z}_k-z_k\|\cdot\|\tilde{z}_i-z_i\|^{p-1} \leq N\|\tilde{Z}-Z\|_p^p.
\end{align}
We combine all the estimates altogether to derive
\begin{align*}
\frac{d}{dt} \|\tilde{Z}-Z\|^p_p &\leq 2p\left(|\kappa_0|+|\kappa_0+2\kappa_1|\right)\|\tilde{Z}-Z\|_p^p,
\end{align*}
or equivalently
\[
\frac{d}{dt} \|\tilde{Z}-Z\|_p \leq 2\left(|\kappa_0|+|\kappa_0+2\kappa_1|\right)\|\tilde{Z}-Z\|_p.
\]
Finally, we apply Gr\"{o}nwall's lemma to get 
\[
\| \tilde{Z}(t) - Z(t)\|_p \leq \exp\big(2t(|\kappa_0|+|\kappa_0+2\kappa_1|)\big)\| \tilde{Z}(0) - Z(0)\|_p.
\]
This yields the first desired estimate:
\[
\sup_{0\leq t\leq T}\| \tilde{Z}(t) - Z(t)\|_p \leq \exp\big(2T(|\kappa_0|+|\kappa_0+2\kappa_1|)\big)\| \tilde{Z}(0) - Z(0)\|_p,
\]

\vspace{0.5cm}

\noindent(ii)~For the second statement, we combine \eqref{Y-0} with the following relations:
\[
\big|\mathrm{Re}\langle \tilde{z}_k-z_k, \tilde{z}_i-z_i \rangle\big| \leq \| \tilde{z}_k - z_k \| \cdot\| \tilde{z}_i - z_i \|,\quad \mathrm{Re}\langle z_i, \tilde{z}_i \rangle - 1 = -\frac{\|z_i - \tilde{z}_i\|^2}{2}
\]
to derive 
\begin{align}\label{ZZ-0}
\begin{aligned}
\frac{d}{dt}\|\tilde{z}_i-z_i\|^2 &\leq \frac{2\kappa_0}{N}\sum_{k=1}^N \left( \|\tilde{z}_k-z_k\| \|\tilde{z}_i-z_i\| - \frac{R_{ik}+\tilde{R}_{ik}}{2}\|z_i - \tilde{z}_i\|^2 ) \right) + \frac{2|\kappa_0+2\kappa_1|}{N}\sum_{k=1}^N {\mathcal{I}}_{i,k}.
\end{aligned}
\end{align}
To estimate the last term, we rewrite \eqref{Y-1} by
\begin{align*}
\mathcal{I}_{i,k} 
&= -(I_{ik}-\tilde{I}_{ik})\mathrm{Im}\langle z_i, \tilde{z}_i \rangle \\
&= -(\mathrm{i}z_i \cdot z_k - \mathrm{i}\tilde{z}_i \cdot \tilde{z}_k)(\mathrm{i}z_i \cdot \tilde{z}_i)\\
&= -\big((\mathrm{i}(z_i-\tilde{z}_i) \cdot (z_k-z_i)) 
+ (\mathrm{i}(z_i - \tilde{z}_i) \cdot z_i)
+ (\mathrm{i}\tilde{z}_i \cdot (z_k - \tilde{z}_k))\big)(\mathrm{i}z_i \cdot \tilde{z}_i) \\
&= -\big((\mathrm{i}(z_i-\tilde{z}_i) \cdot (z_k - z_i))
- (\mathrm{i}\tilde{z}_i \cdot z_i)
+ (\mathrm{i}\tilde{z}_i \cdot (z_k - \tilde{z}_k))\big)(\mathrm{i}z_i \cdot \tilde{z}_i)  \quad (\mbox{by}~~\mathrm{i}z \cdot z = 0) \\
&= -(\mathrm{i}(z_i-\tilde{z}_i) \cdot (z_k-z_i))(\mathrm{i}z_i \cdot \tilde{z}_i)
- (\mathrm{i}\tilde{z}_i \cdot z_i)^2\\
& \hspace{5cm} - (\mathrm{i}(\tilde{z}_i - \tilde{z}_k)\cdot (z_k - \tilde{z}_k)+(\mathrm{i}\tilde{z}_k \cdot (z_k - \tilde{z}_k)) )(\mathrm{i}z_i \cdot \tilde{z}_i)\\
&= -\big(\mathrm{i}(z_i-\tilde{z}_i) \cdot (z_k-z_i) + \mathrm{i}(\tilde{z}_i - \tilde{z}_k) \cdot (z_k - \tilde{z}_k)\big)(\mathrm{i}z_i \cdot \tilde{z}_i) - \underbrace{\big((\mathrm{i} z_i \cdot \tilde{z}_i)-(\mathrm{i}z_k \cdot \tilde{z}_k)\big)(\mathrm{i}z_i \cdot \tilde{z}_i)}_{=: \mathcal{J}_{i,k}}.\\
\end{align*}
In particular, one has
\[
\sum_{i,k=1}^N \mathcal{J}_{i,k} = \sum_{i,k=1}^N (\mathrm{i}z_i \cdot \tilde{z}_i - \mathrm{i}z_k \cdot \tilde{z}_k)^2 \geq 0.
\]
Therefore, as we define $\mathcal{\tilde{I}}_{i,k}:=\mathcal{I}_{i,k}+\mathcal{J}_{i,k}$,
\begin{align}\label{ZZ-0.5}
\sum_{i,k=1}^N\mathcal{I}_{i,k} \leq \sum_{i,k=1}^N-\big(\mathrm{i}(z_i-\tilde{z}_i) \cdot (z_k-z_i) + \mathrm{i}(\tilde{z}_i - \tilde{z}_k) \cdot (z_k - \tilde{z}_k)\big)(\mathrm{i}z_i \cdot \tilde{z}_i) = \sum_{i,k=1}^N\mathcal{\tilde{I}}_{i,k}.
\end{align}
Since $\iota$ is an isometry, for complex vectors $a$ and $b$, we have
\[
| a \cdot b | = | \iota(a) \cdot \iota(b) | \leq \| \iota(a) \|\| \iota(b) \|  = \|a\|\|b\|.
\]
Thus we obtain
\[
| \mathrm{i}z_i \cdot \tilde{z}_i | = | \mathrm{i} z_i \cdot (\tilde{z}_i - z_i) | \leq \| \tilde{z}_i - z_i \|. 
\]
On the other hand, it follows from a priori assumptions that there exist positive constants $A$ and $B$ satisfying
\begin{align} 
\begin{aligned} \label{ZZ-1}
& \max\Big \{  \max_{i,j} \left\| z_i - z_j \right\|,~\max_{k,l}(1-R_{kl}) \Big \} \leq Ae^{-Bt}, \\
&  \max\Big \{  \max_{i,j} \left\| {\tilde z}_i - {\tilde z}_j \right\|,~\max_{k,l}(1-{\tilde R}_{kl}) \Big \} \leq Ae^{-Bt}.
\end{aligned}
\end{align}
Therefore, one has
\begin{align}\label{ZZ-2}
\begin{aligned}
\sum_{i,k} \mathcal{\tilde{I}}_{i,k}  &\leq \sum_{i,k} \big| \mathcal{\tilde{I}}_{i,k} \big| \\
&\leq \sum_{i,k} |-\big(\mathrm{i}(z_i-\tilde{z}_i) \cdot (z_k-z_i) + \mathrm{i}(\tilde{z}_i - \tilde{z}_k) \cdot (z_k - \tilde{z}_k)\big)(\mathrm{i}z_i \cdot \tilde{z}_i) | \\
&\leq \sum_{i,k} (|\big(\mathrm{i}(z_i-\tilde{z}_i) \cdot (z_k-z_i)| + |\big(\mathrm{i}(\tilde{z}_i - \tilde{z}_k) \cdot (z_k - \tilde{z}_k)\big)(\mathrm{i}z_i \cdot \tilde{z}_i) |) \\
& \leq Ae^{-tB}\sum_{i,k}\left(\| z_i - \tilde{z}_i \|^2 + \| z_i - \tilde{z}_i \| \cdot \| z_k - \tilde{z}_k \|\right).
\end{aligned}
\end{align}
To sum up \eqref{ZZ-0} and \eqref{ZZ-1}, we have
\begin{align*}
&\frac{d}{dt}\|\tilde{z}_i-z_i\|^2 \\
&\hspace{0.5cm} \leq \frac{2\kappa_0}{N}\sum_{k=1}^N \left( \|\tilde{z}_k-z_k\| +(Ae^{-Bt}-1)\|z_i - \tilde{z}_i\| ) \right)\|\tilde{z}_i-z_i\| + \frac{2|\kappa_0+2\kappa_1|}{N}\sum_{k=1}^N {\mathcal{I}}_{i,k}.
\end{align*}
Again, we sum up the above relation over all $i \in {\mathcal N}$ to obtain
\begin{align*}
&\frac{d}{dt}\sum_{i=1}^N\|\tilde{z}_i-z_i\|^2\\
&\hspace{0.5cm} \leq \frac{2\kappa_0}{N}\sum_{i,k} \left( \|\tilde{z}_k-z_k\|\|\tilde{z}_i-z_i\| +(Ae^{-Bt}-1)\|z_i - \tilde{z}_i\|^2 ) \right) + \frac{2|\kappa_0+2\kappa_1|}{N}\sum_{i,k}\mathcal{I}_{i,k} \\
&\hspace{0.5cm}\leq \frac{2\kappa_0}{N}\sum_{i,k} \left( \|\tilde{z}_k-z_k\|\|\tilde{z}_i-z_i\| +(Ae^{-Bt}-1)\|z_i - \tilde{z}_i\|^2 ) \right) + \frac{2|\kappa_0+2\kappa_1|}{N}\sum_{i,k}\mathcal{\tilde{I}}_{i,k}  \quad (\mbox{by} ~~\eqref{ZZ-0.5})\\
&\hspace{0.5cm}\leq \frac{2\kappa_0}{N}\sum_{i,k} \left( \|\tilde{z}_k-z_k\|\|\tilde{z}_i-z_i\| +(Ae^{-Bt}-1)\|z_i - \tilde{z}_i\|^2 ) \right)\\
&\hspace{0.7cm} + \frac{2Ae^{-tB}|\kappa_0+2\kappa_1|}{N}\sum_{i,k} (\|z_i - \tilde{z}_i\|^2+\|z_k-\tilde{z}_k\|\|\tilde{z}_i-z_i\|) \quad (\mbox{by}~~ \eqref{ZZ-2}).
\end{align*}
We use \eqref{Y-2} to obtain
\[
\sum_{i, k=1}^N\|\tilde{z}_k-z_k\|\cdot\|\tilde{z}_i-z_i\| \leq N\|\tilde{Z}-Z\|^2.
\]
We collect all the estimates altogether to derive
\begin{align*}
\frac{d}{dt} \|\tilde{Z}-Z\|_2 \leq 2Ae^{-Bt}\left(\kappa_0+2\left|\kappa_0+2\kappa_1|\right)\right\|\tilde{Z}-Z\|^2,
\end{align*}
or equivalently
\begin{align*}
\frac{d}{dt} \|\tilde{Z}-Z\| \leq Ae^{-Bt}\left(\kappa_0+2\left|\kappa_0+2\kappa_1|\right)\right\|\tilde{Z}-Z\|.
\end{align*}
Finally, we apply Gronwall's lemma to get 
\[
\| \tilde{Z}(t) - Z(t)\| \leq \exp{\left(\int_0^t Ae^{-Bs}\left(\kappa_0+2|\kappa_0+2\kappa_1|\right)ds\right)} \| \tilde{Z}(0) - Z(0)\|, \quad t \geq 0.
\]
This completes the proof. 

\vspace{0.5cm}

\section{A formal derivation of the kinetic LHS model} \label{App-C}
\setcounter{equation}{0}
In this appendix, we use the standard BBGKY hierarchy to derive the mean-field kinetic model of the LHS model formally (see \cite{G-H, H-L, Sp} for related results). Here, we briefly summarize the main steps of the derivation for the convenience of readers. \newline

Consider the LHS model:
\begin{align*}
\begin{cases}
\displaystyle \dot{z}_j=\Omega_j z_j+\frac{1}{N}\sum_{k=1}^N\big(\kappa_0(z_k-\langle{z_k, z_j}\rangle z_j)+\kappa_1(\langle{z_j, z_k}\rangle-\langle{z_k, z_j}\rangle)z_j\big), \quad t > 0, \\
\dot{\Omega}_j=0, \quad j \in \mathcal{N}.
\end{cases}
\end{align*}
Let $f^{N} = f^{N}(t,z_1,\Omega_1, z_2, \Omega_2, \cdots, z_N, \Omega_N)$ be the $N$-particle distribution function which is symmetric, i.e., for any permutation $s$ in a symmetric group $S_N$,  we have 
\[f^{N}(t, z_{s(1)},\Omega_{s(1)},z_{s(2)},\Omega_{s(2)}, \cdots, z_{s(N)},\Omega_{s(N)})= f^{N}(t, z_1,\Omega_1, z_2, \Omega_2, \cdots, z_N, \Omega_N).
\]
Then, it satisfies 
 \begin{equation} \label{Liu}
\partial_t f^N +\sum_{i=1}^{N}\nabla_{z_i} \cdot\left[  \left(\Omega_i z_i+\frac{1}{N}\sum_{k=1}^{N}\kappa_0(z_k-\langle{z_k, z_i}\rangle z_i)+\kappa_1(\langle{z_i, z_k}\rangle-\langle{z_k, z_i}\rangle)z_i \right)f^N\right ] =0.
\end{equation}
For notational convenience, we write
\[
Z_i:=(z_i, \Omega_i),\quad {\mathcal Z}=(Z_1, Z_2, \cdots, Z_N),
\]
and set the measure $dZ_i$ on $\Xi$ as follows:
\[
dZ_i:=d\sigma_{z_i}\otimes d\Omega_i\quad\text{and}\quad dZ^i:=\prod_{\substack{k=1\\k\neq i}}^N dZ_k.
\]
Then we can simply write
\[
f^N(t,{\mathcal Z})=f^N(t,Z_1, Z_2, \cdots, Z_N)= f^{N}(t,z_1,\Omega_1, z_2, \Omega_2, \cdots, z_N, \Omega_N).
\]
Now we integrate \eqref{Liu} with respect to variables  $Z_2, \cdots, Z_N$ to get
\begin{align}
\begin{aligned}\label{BB1}
0&=\partial_t \int_{\Xi^{N-1}} f^N(t,\mathcal{Z})dZ^1 \\
&+ \sum_{i=1}^{N}\int_{\Xi^{N-1}}\nabla_{z_i}\cdot 
\bigg(\bigg(\Omega_i z_i+\frac{1}{N}\sum_{k=1}^{N}\kappa_0(z_k-\langle{z_k, z_i}\rangle z_i)+\kappa_1(\langle{z_i, z_k}\rangle-\langle{z_k, z_i}\rangle)z_i\bigg)f^N(t,\mathcal{Z})\bigg)dZ^1.
\end{aligned}
\end{align}
Then, by divergence theorem on $\mathbb{HS}^{d-1}$, we have
\begin{align}
\begin{aligned}\label{BBB1}
&\int_{\mathbb{HS}^{d-1}}\nabla_{z_i}\cdot \bigg(\Omega_i z_i+\frac{1}{N}\sum_{k=1}^{N}\kappa_0(z_k-\langle{z_k, z_i}\rangle z_i)+\kappa_1(\langle{z_i, z_k}\rangle-\langle{z_k, z_i}\rangle)z_if^N(t,z_1,z_2, \cdots, z_N)\bigg)dz_i \\
&\hspace{1cm} =0, \qquad 2 \leq i \leq N.
\end{aligned}
\end{align}
If we integrate above relation \eqref{BBB1} over $\mathrm{Skew}_d\bbc\times \Xi^{N-2}$ with respect to the product measure $\displaystyle d\Omega_i\prod_{\substack{\ell=2\\\ell\neq i}}dZ_l$, we have the following relation:
\[
\int_{\Xi^{N-1}}\nabla_{z_i}\cdot \bigg(\Omega_i z_i+\frac{1}{N}\sum_{k=1}^{N}\kappa_0(z_k-\langle{z_k, z_i}\rangle z_i)+\kappa_1(\langle{z_i, z_k}\rangle-\langle{z_k, z_i}\rangle)z_if^N(t,z_1,z_2, \cdots, z_N)\bigg)dZ^1=0.
\]
Then, this relation simplifies the R.H.S. of \eqref{BB1} as follows:
\begin{align*}
\begin{aligned}
&\text{R.H.S. of \eqref{BB1}} \\
&=\partial_t \int_{\Xi^{N-1}} f^N(t,\mathcal{Z})dZ^1 \\
&\hspace{0.2cm}+ \int_{\Xi^{N-1}}\nabla_{z_1}\cdot 
\bigg(\bigg(\Omega_1 z_1+\frac{1}{N}\sum_{k=1}^{N}\big(\kappa_0(z_k-\langle{z_k, z_1}\rangle z_1)+\kappa_1(\langle{z_1, z_k}\rangle-\langle{z_k, z_1}\rangle)z_1\big)\bigg)f^N(t,\mathcal{Z})\bigg)dZ^1\\
&=\underbrace{\partial_t \int_{\Xi^{N-1}} f^N(t, \mathcal{Z})dZ^1}_{=:\mathcal{I}_1} +\underbrace{\int_{\Xi^{N-1}}\nabla_{z_1}\cdot\bigg((\Omega_1z_1) f^N(t,\mathcal{Z})\bigg)dZ^1}_{=:\mathcal{I}_2}\\
&\hspace{0.2cm}+\underbrace{\frac{1}{N}\sum_{k=1}^N\int_{\Xi^{N-1}}\nabla_{z_1}\cdot\bigg(\big(\kappa_0(z_k-\langle{z_k, z_1}\rangle z_1)+\kappa_1(\langle{z_1, z_k}\rangle-\langle{z_k, z_1}\rangle)z_1\big)f^N(t,\mathcal{Z})\bigg)dZ^1}_{=:\mathcal{I}_3}.
\end{aligned}
\end{align*}
On the other hand, by the symmetry of $f^N$, we get
\begin{align}\label{BB1-1}
f^N(t,Z_1, Z_2, \cdots, Z_k, \cdots, Z_N)=f^N(t, Z_1, Z_k, \cdots, Z_2,\cdots, Z_N).
\end{align}
This yields 
\begin{align*}
&\int_{\Xi^{N-1}}\nabla_{z_1}\cdot\bigg(\big(\kappa_0(z_k-\langle{z_k, z_1}\rangle z_1)+\kappa_1(\langle{z_1, z_k}\rangle-\langle{z_k, z_1}\rangle)z_1\big)f^N(t,Z_1, Z_2, \cdots, Z_k, \cdots, Z_N)\bigg)dZ^1\\
=&\int_{\Xi^{N-1}}\nabla_{z_1}\cdot\bigg(\big(\kappa_0(z_2-\langle{z_2, z_1}\rangle z_1)+\kappa_1(\langle{z_1, z_2}\rangle-\langle{z_2, z_1}\rangle)z_1\big)f^N(t,Z_1, Z_k, \cdots, Z_2,\cdots, Z_n)\bigg)dZ^1\\
=&\int_{\Xi^{N-1}}\nabla_{z_1}\cdot\bigg(\big(\kappa_0(z_2-\langle{z_2, z_1}\rangle z_1)+\kappa_1(\langle{z_1, z_2}\rangle-\langle{z_2, z_1}\rangle)z_1\big)f^N(t,Z_1, Z_2, \cdots, Z_k, \cdots, Z_N)\bigg)dZ^1.
\end{align*}
We sum up the above relation over all $k$ to get 
\begin{align}
\begin{aligned}\label{BB1-2}
&\sum_{k=1}^N\int_{\Xi^{N-1}}\nabla_{z_1}\cdot\bigg(\big(\kappa_0(z_k-\langle{z_k, z_1}\rangle z_1)+\kappa_1(\langle{z_1, z_k}\rangle-\langle{z_k, z_1}\rangle)z_1\big)f^N(t,\mathcal{Z})\bigg)dZ^1\\
=&\sum_{k=1}^N\int_{\Xi^{N-1}}\nabla_{z_1}\cdot\bigg(\big(\kappa_0(z_2-\langle{z_2, z_1}\rangle z_1)+\kappa_1(\langle{z_1, z_2}\rangle-\langle{z_2, z_1}\rangle)z_1\big)f^N(t,\mathcal{Z})\bigg)dZ^1\\
=&(N-1)\int_{\Xi^{N-1}}\nabla_{z_1}\cdot\bigg(\big(\kappa_0(z_2-\langle{z_2, z_1}\rangle z_1)+\kappa_1(\langle{z_1, z_2}\rangle-\langle{z_2, z_1}\rangle)z_1\big)f^N(t,\mathcal{Z})\bigg)dZ^1.
\end{aligned}
\end{align}
Now, we set $f^{N:k}(t,Z_1,\cdots, Z_k)$ to be the marginal marginal distribution function of $f^N$ defined as follows:
\[
f^{N:k}(t,Z_1,\cdots, Z_k) := \int_{(\Xi)^{N-k}} f^N(t,Z_1,\cdots,Z_N) \prod_{\ell=k+1}^NdZ_\ell.
\]
From this definition, we can obtain the following:
\begin{align}
\begin{aligned}\label{BB1-3}
&\int_{\Xi^{N-1}}\nabla_{z_1}\cdot\bigg(\big(\kappa_0(z_2-\langle{z_2, z_1}\rangle z_1)+\kappa_1(\langle{z_1, z_2}\rangle-\langle{z_2, z_1}\rangle)z_1\big)f^N(t,{\mathcal Z})\bigg)dZ^1\\
=&\int_{\Xi^{N-1}}\nabla_{z_1}\cdot\bigg(\big(\kappa_0(z_2-\langle{z_2, z_1}\rangle z_1)+\kappa_1(\langle{z_1, z_2}\rangle-\langle{z_2, z_1}\rangle)z_1\big)f^N(t, {\mathcal Z})\bigg)dZ_2dZ_3\cdots dZ_N\\
=&\int_{\Xi}\nabla_{z_1}\cdot\bigg(\big(\kappa_0(z_2-\langle{z_2, z_1}\rangle z_1)+\kappa_1(\langle{z_1, z_2}\rangle-\langle{z_2, z_1}\rangle)z_1\big)f^{N:2}(t,Z_1, Z_2)\bigg)dZ_2.
\end{aligned}
\end{align}
Below, we  estimate $\mathcal{I}_\alpha$ one by one. \newline

\noindent$\diamond$ (Estimate or $\mathcal{I}_1$):~Note that 
\begin{align}\label{BB2-1}
\begin{aligned}
\mathcal{I}_1&=\partial_t \int_{\Xi^{N-1}} f^N(t, {\mathcal Z})dZ^1=\partial_t \int_{\Xi^{N-1}} f^N(t,Z_1, Z_2, \cdots, Z_N)dZ_2dZ_3\cdots dZ_N\\
&=\partial_t \int_{\Xi} f^{N:2}(t,Z_1, Z_2)dZ_2.
\end{aligned}
\end{align}
\vspace{0.2cm}

\noindent$\diamond$ (Estimate or $\mathcal{I}_2$):~Similarly, we have 
\begin{align}\label{BB2-2}
\begin{aligned}
\mathcal{I}_2&=\int_{\Xi^{N-1}}\nabla_{z_1}\cdot\bigg((\Omega_1z_1) f^N(t,Z)\bigg)dZ^1 =\int_{\Xi^{N-1}}\nabla_{z_1}\cdot\bigg((\Omega_1z_1) f^N(t,Z_1,Z_2,\cdots, Z_N)\bigg)dZ_2\cdots dZ_N\\
&=\int_{\Xi}\nabla_{z_1}\cdot\bigg((\Omega_1z_1) f^{N:2}(t,Z_1,Z_2)\bigg)dZ_2.
\end{aligned}
\end{align}
\vspace{0.2cm}

\noindent$\diamond$ (Estimate or $\mathcal{I}_3$):~If we apply \eqref{BB1-1}, \eqref{BB1-2}, and \eqref{BB1-3} step by step, we have
\begin{align}\label{BB2-3}
\begin{aligned}
\mathcal{I}_3&=\frac{1}{N}\sum_{k=1}^N\int_{\Xi^{N-1}}\nabla_{z_1}\cdot\bigg(\big(\kappa_0(z_k-\langle{z_k, z_1}\rangle z_1)+\kappa_1(\langle{z_1, z_k}\rangle-\langle{z_k, z_1}\rangle)z_1\big)f^N(t,Z)\bigg)dZ^1\\
&=\left(\frac{N-1}{N}\right)\int_{\Xi^{N-1}}\nabla_{z_1}\cdot\bigg(\big(\kappa_0(z_2-\langle{z_2, z_1}\rangle z_1)+\kappa_1(\langle{z_1, z_2}\rangle-\langle{z_2, z_1}\rangle)z_1\big)f^N(t,Z)\bigg)dZ^1\\
&=\left(\frac{N-1}{N}\right)\int_{\Xi}\nabla_{z_1}\cdot\bigg(\big(\kappa_0(z_2-\langle{z_2, z_1}\rangle z_1)+\kappa_1(\langle{z_1, z_2}\rangle-\langle{z_2, z_1}\rangle)z_1\big)f^{N:2}(t,Z_1, Z_2)\bigg)dZ_2.
\end{aligned}
\end{align}
\vspace{0.2cm}

\noindent It follows from \eqref{BB2-1}, \eqref{BB2-2}, and \eqref{BB2-3} that 
\begin{align*}
\begin{aligned}
\mathcal{I}_1+\mathcal{I}_2+\mathcal{I}_3 &= \partial_t \int_{\Xi} f^{N:2}(t,Z_1, Z_2)dZ_2+\int_{\Xi}\nabla_{z_1}\cdot\bigg((\Omega_1z_1) f^{N:2}(t,Z_1,Z_2)\bigg)dZ_2\\
&\hspace{0.2cm}+\left(\frac{N-1}{N}\right)\int_{\Xi}\nabla_{z_1}\cdot\bigg(\big(\kappa_0(z_2-\langle{z_2, z_1}\rangle z_1)+\kappa_1(\langle{z_1, z_2}\rangle-\langle{z_2, z_1}\rangle)z_1\big)f^{N:2}(t,Z_1, Z_2)\bigg)dZ_2\\
&=0.
\end{aligned}
\end{align*}
Now we take the mean-field limit $N \to \infty$ and obtain $f^1 := \displaystyle\lim_{N \to \infty}f^{N:1}(t,Z_1)$ and $f^2 := \displaystyle\lim_{N \to \infty}f^{N:2}(t, Z_1, Z_2)$ which satisfy 
\begin{align}\label{MC1}
\begin{aligned}
&\partial_t \int_{\Xi} f^2(t,Z_1, Z_2)dZ_2+\int_{\Xi}\nabla_{z_1}\cdot\bigg((\Omega_1z_1) f^2(t,Z_1,Z_2)\bigg)dZ_2\\
&+ \int_{\Xi}\nabla_{z_1}\cdot\bigg(\big(\kappa_0(z_2-\langle{z_2, z_1}\rangle z_1)+\kappa_1(\langle{z_1, z_2}\rangle-\langle{z_2, z_1}\rangle)z_1\big)f^2(t,Z_1, Z_2)\bigg)dZ_2 =0.
\end{aligned}
\end{align} 
Finally, from the molecular chaos assumption which implies that $f^2(t,Z_1,Z_2) = f^1(t,Z_1)f^1(t,Z_2)$ as $N \to \infty$, it follows that 
\begin{align*}
\eqref{MC1}&= \partial_t \int_{\Xi} f^1(t,Z_1)f^1(t,Z_2)dZ_2 + \int_{\Xi}\nabla_{z_1}\cdot\bigg((\Omega_1z_1) f^1(t,Z_1)f^1(t,Z_2)\bigg)dZ_2 \\
&\hspace{0.2cm}+\int_{\Xi}\nabla_{z_1}\cdot\bigg(\big(\kappa_0(z_2-\langle{z_2, z_1}\rangle z_1)+\kappa_1(\langle{z_1, z_2}\rangle-\langle{z_2, z_1}\rangle)z_1\big)f^1(t,Z_1)f^1(t,Z_2)\bigg)dZ_2\\
&=\partial_t f^1(t,Z_1) + \nabla_{z_1}\cdot\big((\Omega_1z_1) f^1(t,Z_1)\big)\\
&\hspace{0.2cm}+f^1(t,Z_1)\int_{\Xi}\nabla_{z_1}\cdot\bigg(\big(\kappa_0(z_2-\langle{z_2, z_1}\rangle z_1)+\kappa_1(\langle{z_1, z_2}\rangle-\langle{z_2, z_1}\rangle)z_1\big)f^1(t, Z_2)\bigg)dZ_2\\
&=0 .
\end{align*}
In conclusion, one-particle distribution function $f(t,z)$ satisfies the following equation :
\begin{align*}
\begin{cases}
\displaystyle 0=\partial_t f(t,z, \Omega)+\nabla_{z}\cdot(L[f](t,z, \Omega)f(t,z, \Omega)),\vspace{0.2cm}\\
\displaystyle L[f](t,z, \Omega)=\Omega z+\int_{\Xi}\big(\kappa_0(z_*-\langle{z_*, z}\rangle z)+\kappa_1(\langle{z, z_*}\rangle-\langle{z_*, z}\rangle)z\big) f(t,z_*, \Omega_*)dz_*d\Omega_*.
\end{cases}
\end{align*}
This is consistent with the kinetic LHS model \eqref{C-1}.

\vspace{0.5cm}

\section{Proof of the second statement of Lemma \ref{L4.4}}\label{App-D}
\setcounter{equation}{0}
Recall that the order parameter $R$ is given by
\[
R^2=J_\rho\cdot J_\rho, \quad J_\rho :=  \int_{\mathbb{HS}^{d-1}}z\rho(t,z)d\sigma_z.
\]
This yields
\[
\frac{d^2 R^2}{dt^2} =\frac{d}{dt}\left(2J_\rho\cdot \frac{dJ_\rho}{dt} \right)=2\left\|\frac{dJ_\rho}{dt} \right\|^2+2J_\rho\cdot \frac{d^2J_\rho}{dt^2}.
\]
Next, we estimate the first and second derivatives of $J_\rho$.  Let $e\in\mathbb{HS}^{d-1}$ be an arbitrary unit vector. \newline

\noindent$\bullet$~Case A (Boundedness of $\Big| \frac{dJ_\rho}{dt} \Big|)$:~It follows from \eqref{est order} that
\[
e\cdot\frac{d J_\rho}{dt}= \kappa_0 \int_{\mathbb{HS}^{d-1}}\mathbb{P}_{z^\perp}(e) \cdot \mathbb{P}_{z^\perp}(J_\rho )\rho(t,z)d\sigma_z + (\kappa_0 + 2\kappa_1)\int_{\mathbb{HS}^{d-1}}\mathbb{P}_{z^\perp}(e) \cdot \mathbb{P}_{\mathrm{i}z}(J_\rho )\rho(t,z)d\sigma_z.
\]
Then, we use $\|J_\rho\|\leq 1$ to obtain the following inequalities:
\begin{align*}
\begin{aligned}
& |\mathbb{P}_{z^\perp}(e) \cdot \mathbb{P}_{z^\perp}(J_\rho)|\leq \|\mathbb{P}_{z^\perp}(e)\|\cdot\|\mathbb{P}_{z^\perp}(J_\rho)\|\leq \|e\|\cdot \|J_\rho\|\leq1, \\
& |\mathbb{P}_{z^\perp}(e) \cdot \mathbb{P}_{\mathrm{i}z}(J_\rho )|=\|\mathbb{P}_{z^\perp}(e) \|\cdot \|\mathbb{P}_{\mathrm{i}z}(J_\rho )\|\leq \|e\|\cdot\|J_\rho\|\leq 1.
\end{aligned}
\end{align*}

\noindent These yield
\begin{align*}
\left|e\cdot\frac{d J_\rho}{dt} \right| &\leq  \kappa_0 \int_{\mathbb{HS}^{d-1}}|\mathbb{P}_{z^\perp}(e) \cdot \mathbb{P}_{z^\perp}(J_\rho)|\rho(t,z)d\sigma_z 
+ (\kappa_0 + 2\kappa_1)\int_{\mathbb{HS}^{d-1}}|\mathbb{P}_{z^\perp}(e) \cdot \mathbb{P}_{\mathrm{i}z}(J_\rho )
|\rho(t,z)d\sigma_z\\
&\leq  \kappa_0 \int_{\mathbb{HS}^{d-1}}\rho(t,z)d\sigma_z + (\kappa_0 + 2\kappa_1)\int_{\mathbb{HS}^{d-1}}\rho(t,z)d\sigma_z =2(\kappa_0+\kappa_1).
\end{align*}
Since $e$ was an arbitrary unit vector, one has 
\[
\left\|\frac{d J_\rho}{dt} \right\|\leq 2(\kappa_0+\kappa_1).
\]
\vspace{0.2cm}

\noindent$\bullet$~Case B (Boundedness of $\frac{d^2J_\rho}{dt^2}$):~It follows from \eqref{est order} that 
\begin{align*}
e\cdot\frac{d^2 J_\rho}{dt^2} &= \kappa_0 \underbrace{\frac{d}{dt}\left(\int_{\mathbb{HS}^{d-1}}\mathbb{P}_{z^\perp}(e) \cdot \mathbb{P}_{z^\perp}(J_\rho )\rho(t,z)d\sigma_z\right)}_{=:\mathcal{K}_1}  \\
&\hspace{5cm}+(\kappa_0 + 2\kappa_1)\frac{d}{dt}\underbrace{\left(\int_{\mathbb{HS}^{d-1}}\mathbb{P}_{z^\perp}(e) \cdot \mathbb{P}_{\mathrm{i}z}(J_\rho )\rho(t,z)d\sigma_z\right)}_{=:\mathcal{K}_2}.
\end{align*}
Since $J_\rho$ and $\rho$ depend on time $t$, we have
\begin{align*}
\begin{aligned}
\mathcal{K}_1 &=\underbrace{\int_{\mathbb{HS}^{d-1}}\mathbb{P}_{z^\perp}(e) \cdot (\partial_t\mathbb{P}_{z^\perp}(J_\rho ))\rho(t,z)d\sigma_z}_{=:\mathcal{L}_1}
+\underbrace{\int_{\mathbb{HS}^{d-1}}\mathbb{P}_{z^\perp}(e) \cdot \mathbb{P}_{z^\perp}(J_\rho )\partial_t\rho(t,z)d\sigma_z}_{=:\mathcal{L}_2}, \\
\mathcal{K}_2 &=\underbrace{\int_{\mathbb{HS}^{d-1}}\mathbb{P}_{z^\perp}(e) \cdot (\partial_t\mathbb{P}_{\mathrm{i}z}(J_\rho ))\rho(t,z)d\sigma_z}_{=:\mathcal{L}_3}
+\underbrace{\int_{\mathbb{HS}^{d-1}}\mathbb{P}_{z^\perp}(e) \cdot \mathbb{P}_{\mathrm{i}z}(J_\rho )\partial_t\rho(t,z)d\sigma_z}_{=:\mathcal{L}_4}.
\end{aligned}
\end{align*}

\vspace{0.1cm}

\noindent$\diamond$ Case B.1 (Estimate of $\mathcal{L}_1$ and $\mathcal{L}_3$):~Since $\mathbb{P}_{z^\perp}(J_\rho)$ and $\mathbb{P}_{\mathrm{i}z}(J_\rho)$ are continuously differentiable functions defined on the compact domain $\mathbb{B}^{2d}:=\{z\in\bbc^d: \|z\|\leq 1\}$,  $\nabla_{\bbc^d} \mathbb{P}_{z^\perp},~\nabla_{\bbc^d}\mathbb{P}_{\mathrm{i}z}$
are bounded on $\mathbb{B}^{2d}$.  i.e. there exists a positive constant $M$ such that 
\[
\|\nabla_{\bbc^d} \mathbb{P}_{z^\perp}\|_{op}\leq M,\quad \|\nabla_{\bbc^d}\mathbb{P}_{\mathrm{i}z}\|_{op}\leq M
\]
for all $z\in\mathbb{HS}^{d-1}$. This yields
\begin{align*}
\|\partial_t\mathbb{P}_{z^\perp}(J_\rho)\|=\left\|\nabla_{\bbc^d} \mathbb{P}_{z^\perp} \left(\frac{d J_\rho}{dt} \right)\right\|\leq \|\nabla_{\bbc^d} \mathbb{P}_{z^\perp}\|_{op}\cdot \left\|\frac{d J_\rho}{dt} \right\|\leq 2M(\kappa_0+\kappa_1)
\end{align*}
and
\begin{align*}
\|\partial_t\mathbb{P}_{\mathrm{i}z}(J_\rho)\|=\left\|\nabla_{\bbc^d}\mathbb{P}_{\mathrm{i}z}\left(\frac{d J_\rho }{dt} \right)\right\|\leq \|\nabla_{\bbc^d}\mathbb{P}_{\mathrm{i}z}\|_{op}\cdot \left\|\frac{d J_\rho}{dt} \right\|\leq 2M(\kappa_0+\kappa_1).
\end{align*}
Finally, we have
\begin{align*}
|\mathcal{L}_1|&\leq \left| \int_{\mathbb{HS}^{d-1}}\mathbb{P}_{z^\perp}(e) \cdot (\partial_t\mathbb{P}_{z^\perp}(J_\rho ))\rho(t,z)d\sigma_z\right|
\leq\int_{\mathbb{HS}^{d-1}}\|\mathbb{P}_{z^\perp}(e) \|\cdot \left\|\partial_t\mathbb{P}_{z^\perp}(J_\rho )\right\|\rho(t,z)d\sigma_z\\
&\leq 2M(\kappa_0+\kappa_1)\int_{\mathbb{HS}^{d-1}}\rho(t,z)d\sigma_z=2M(\kappa_0+\kappa_1)
\end{align*}
and
\begin{align*}
|\mathcal{L}_3|&\leq \left|\int_{\mathbb{HS}^{d-1}}\mathbb{P}_{z^\perp}(e) \cdot (\partial_t\mathbb{P}_{\mathrm{i}z}(J_\rho ))\rho(t,z)d\sigma_z\right|
\leq\int_{\mathbb{HS}^{d-1}}\|\mathbb{P}_{\mathrm{i}z}(e) \|\cdot \left\|\partial_t\mathbb{P}_{\mathrm{i}z}(J_\rho )\right\|\rho(t,z)d\sigma_z\\
&\leq 2M(\kappa_0+\kappa_1)\int_{\mathbb{HS}^{d-1}}\rho(t,z)d\sigma_z=2M(\kappa_0+\kappa_1).
\end{align*}

\vspace{0.1cm}

\noindent$\diamond$ Case B.2 (Estimate of $\mathcal{L}_2$):~ It follows from \eqref{D-1} that 
\begin{align*}
\mathcal{L}_2&=\int_{\mathbb{HS}^{d-1}}\mathbb{P}_{z^\perp}(e) \cdot \mathbb{P}_{z^\perp}(J_\rho )\partial_t\rho(t,z)d\sigma_z
=-\int_{\mathbb{HS}^{d-1}}\mathbb{P}_{z^\perp}(e) \cdot \mathbb{P}_{z^\perp}(J_\rho ))\big(\nabla_z\cdot(L[\rho]\rho)\big)d\sigma_z.
\end{align*}
By Remark \ref{R4.2}, we have
\[
\mathcal{L}_2=\int_{\mathbb{HS}^{d-1}}\nabla_z\big(\mathbb{P}_{z^\perp}(e) \cdot \mathbb{P}_{z^\perp}(J_\rho ))\big)\cdot(L[\rho]\rho)d\sigma_z.
\]
It follows from Lemma \ref{L3.2} that 
\begin{align*}
\mathbb{P}_{z^\perp}(e) \cdot \mathbb{P}_{z^\perp}(J_\rho )&=(e-(z\cdot e)z)\cdot(J_\rho-(z\cdot J_\rho)z)\\
&=e\cdot J_\rho-(z\cdot e)(z\cdot J_\rho)-(z\cdot J_\rho)(e\cdot z)+(z\cdot e)(z\cdot J_\rho)(z\cdot z)\\
&=e\cdot J_\rho-(z\cdot e)(z\cdot J_\rho).
\end{align*}
This yields,
\begin{align*}
\|\nabla_z\big(\mathbb{P}_{z^\perp}(e) \cdot \mathbb{P}_{z^\perp}(J_\rho )\big)\|&=\|\mathbb{P}_z \nabla_{\bbc^d}\big(e\cdot J_\rho-(z\cdot e)(z\cdot J_\rho)\big)\|\\
&\leq \|\nabla_{\bbc^d}\big(e\cdot J_\rho-(z\cdot e)(z\cdot J_\rho)\big)\|=\|\nabla_{\bbc^d}(z\cdot e)(z\cdot J_\rho)\|\\
&=\|(z\cdot J_\rho)e+(z\cdot e)J_\rho\|\leq 2.
\end{align*}
Here we used $\|e\|, \|J_\rho\|, \|z\|\leq 1$ and the triangle inequality. \newline

Finally we have
\begin{align*}
|\mathcal{L}_2|&\leq2\int_{\mathbb{HS}^{d-1}}\|L[\rho]\|\rho(t, z) d\sigma_z\\
&=2\int_{\mathbb{HS}^{d-1}}\| \kappa_0(J_\rho - ( J_\rho \cdot z) z) + (\kappa_0+2\kappa_1)((\mathrm{i}z)\cdot J_\rho)(\mathrm{i}z)\|\rho(t, z)d\sigma_z\\
&\leq2\int_{\mathbb{HS}^{d-1}}\big(\kappa_0\| J_\rho - ( J_\rho \cdot z) z\| +(\kappa_0+2\kappa_1) \|((\mathrm{i}z)\cdot J_\rho)(\mathrm{i}z)\|\big)\rho(t, z)d\sigma_z\\
&\leq 2\kappa_0+2(\kappa_0+2\kappa_1)=4(\kappa_0+\kappa_1).
\end{align*}

\vspace{0.2cm}

\noindent$\diamond$ Case B.3 (Estimate of $\mathcal{L}_4$):~Similar to the previous cases, $|\mathcal{L}_4|$ is bounded as follows:
\[
|\mathcal{L}_4|\leq 4(\kappa_0+\kappa_1).
\]
If we sum-up all the estimates of $\mathcal{L}_j$ ($j=1, 2, 3, 4$), we have
\begin{align*}
\left|e\cdot\frac{d^2 J_\rho}{dt^2} \right|&\leq \kappa_0|\mathcal{K}_1|+(\kappa_0+2\kappa_1)|\mathcal{K}_2|\\
&\leq \kappa_0(|\mathcal{L}_1|+|\mathcal{L}_2|)+(\kappa_0+2\kappa_1)(|\mathcal{L}_3|+|\mathcal{L}_4|)\\
&\leq \kappa_0(2M(\kappa_0+\kappa_1)+4(\kappa_0+\kappa_1))+(\kappa_0+2\kappa_1)(2M(\kappa_0+\kappa_1)+4(\kappa_0+\kappa_1))\\
&=4(\kappa_0+\kappa_1)^2(M+2).
\end{align*}
Since above inequality holds for any unit vector $e\in\mathbb{HS}^{d-1}$, we have
\[
\left\|\frac{d^2 J_\rho}{dt^2} \right\|\leq 4(\kappa_0+\kappa_1)^2(M+2).
\]
Finally, we have a uniform boundedness for the second derivative of $R$:
\begin{align*}
\frac{d^2 R^2}{dt^2}&=2\left\|\frac{d J_\rho}{dt} \right\|^2+2J_\rho\cdot \frac{d^2 J_\rho}{dt^2} \leq  \big( 2(\kappa_0+\kappa_1)\big)^2+4(\kappa_0+\kappa_1)^2(M+2)\\
&=4(\kappa_0+\kappa_1)^2(M+3).
\end{align*}
This completes the proof.

\end{document}